\documentclass[12pt, reqno]{amsart}
\usepackage{amsmath, amsthm, amscd, amsfonts, amssymb, graphicx, color}
\usepackage[bookmarksnumbered, colorlinks, plainpages]{hyperref}
\hypersetup{colorlinks=true,linkcolor=red, anchorcolor=green, citecolor=cyan, urlcolor=red, filecolor=magenta, pdftoolbar=true}

\DeclareMathSymbol{\subsetneqq}{\mathbin}{AMSb}{36}

\newcommand{\R}{\mathbb{R}}

\newcommand{\C}{\mathbb{C}}

\newcommand{\beq}{\begin{eqnarray}}
\newcommand{\eeq}{\end{eqnarray}}
\newcommand{\bq}{\begin{equation}}
\newcommand{\eq}{\end{equation}}
\newcommand{\beqn}{\begin{eqnarray*}}
\newcommand{\eeqn}{\end{eqnarray*}}
\newcommand{\bex}{\begin{exo}}
\newcommand{\eex}{\end{exo}}
\newcommand{\ben}{\begin{enumerate}}
\newcommand{\een}{\end{enumerate}}

\newtheorem{th1}{{\bf Theorem}}[section]
\newtheorem{thm}[th1]{{\bf Theorem}}
\newtheorem{lem}[th1]{{\bf Lemma}}
\newtheorem{prop}[th1]{{\bf Proposition}}
\newtheorem{cor}[th1]{{\bf Corollary}}
\newtheorem{rem}[th1]{\bf Remark}
\newtheorem{rems}[th1]{\bf Remarks}
\newtheorem{defi}[th1]{\bf Definition}

\author[T. Saanouni]{Tarek Saanouni}
\address{Departement of Mathematics, College of Science and Arts in Uglat Asugour, Qassim University, Buraydah, Kingdom of Saudi Arabia.}
\email{\sl t.saanouni@qu.edu.sa}
\email{\sl Tarek.saanouni@ipeiem.rnu.tn}

\title[Choquard INLS]{Scattering theory for a class of radial focusing inhomogeneous Hartree equations}

\date{\today}
\begin{document}
\begin{abstract}
This paper studies the asymptotic behavior of global solutions to the generalized Hartree equation 
$$i\dot u+\Delta  u+(I_\alpha *|\cdot|^b|u|^p)|x|^b|u|^{p-2}u=0 .$$
Indeed, using a new approach due to \cite{dm}, one proves the scattering of the above inhomogeneous Choquard equation in the mass-super-critical and energy sub-critical regimes with radial setting.
\end{abstract}
\maketitle
%\tableofcontents
%\address{University Tunis El Manar, Faculty of Sciences of Tunis, Department of Mathematics.}%\tableofcontents
\vspace{ 1\baselineskip}
\renewcommand{\theequation}{\thesection.\arabic{equation}}
%\newpage
%%%%%%%%%%%%%%%%%%%%%%%%%%%%%%%%%%%%%%%%%%%%%%%%%%%%%%%%%%%%%%%%%%%%%%%%%%%%%%%%%%%%%%%%%
\section{Introduction}
%%%%%%%%%%%%%%%%%%%%%%%%%%%%%%%%%%%%%%%%%%%%%%%%%%%%%%%%%%%%%%%%%%%%%%%%%%%%%%%%%%%%%%%%%%%%
%\vspace{2cm} \newpage
%\begin{center}
%\large{ \bf A note on fourth-order for a coupled system of Schr\"odinger equations}
%\end{center}
This note is concerned with the energy scattering of the Cauchy problem for the following inhomogeneous generalized Hartree equation %with exponential-type non-linearity
\begin{equation}
\left\{
\begin{array}{ll}
i\dot u+\Delta  u+(I_\alpha *|\cdot|^b|u|^p)|x|^b|u|^{p-2}u=0 ;\\
u(0,.)=u_0,
\label{S}
\end{array}
\right.
\end{equation}
where $u: \R\times\R^N \to \C$, for some $N\geq3$. The unbounded inhomogeneous term is $|\cdot|^b$, for some $b\neq0$  and the Riesz-potential is defined on $\R^N$ by 
$$I_\alpha:=\frac{\Gamma(\frac{N-\alpha}2)}{\Gamma(\frac\alpha2)\pi^\frac{N}22^\alpha|\cdot|^{N-\alpha}}.$$%:=\frac{\mathcal K}{|\cdot|^{N-\alpha}},\quad  0<\alpha<N.$$
Here and hereafter, one assumes the following conditions done in \cite{saa}, 
\begin{equation}\label{cnd}
\min\{-b,\alpha,N-\alpha,N+b,4+\alpha+2b-N\}>0.
\end{equation}
Some particular cases of the above equation model many physical phenomena. For instance, it arises in the study of the mean-field limit of large systems of non-relativistic bosonic atoms and molecules \cite{fl}. Moreover, it describes the propagation of electromagnetic waves in plasmas \cite{bc}. The fourth space dimensional case gives the so-called Schr\"odinger-Poisson equation, which appeares in semi-conductor and quantum mechanics theories \cite{drz}.\\

%, i.e., of a regime where the number of bosons is very large, but the interactions
%between them are weak \cite{fl,hs}}.\\
%The homogeneous problem associated to the above problem $(b=0)$ has several physical origins such as quantum mechanics \cite{lr} (Hartree-Fock theory and non-relativistic quantum theory \cite{ll,pg}). If $b=0$ and $p=2$, then \eqref{S} models the dynamics of (pseudo-relativistic) boson stars, where the potential is the Newtonian gravitational potential in the appropriate physical units \cite{es,lenz}. The inhomogeneous non-linear Schr\"odinger equation 
%$$ i\dot u+\Delta u+k(x) |u|^{p-1}u=0,$$
%models the beam propagation in an inhomogeneous medium, where $u$ is the electric field in laser optics and $k$ is proportional to the electric density \cite{Gill, Liu}.\\

The inhomogeneous Choquard equation \eqref{S} satisfies the time-space scaling invariance
$$u_\lambda=\lambda^\frac{2+2b+\alpha}{2(p-1)}u(\lambda^{2}.,\lambda .),\quad\lambda>0.$$
%Indeed, one has%
The unique homogeneous Sobolev norm invariant under the previous space scaling appears in the following identity% Computing the homogeneous Sobolev norm,
$$\|u_\lambda(t)\|_{\dot H^\mu}=\lambda^{\mu-s_c}\|u(\lambda^2t)\|_{\dot H^\mu},$$
where the critical Sobolev exponent is
$$s_c:=\frac N2-\frac{2+2b+\alpha}{2(p-1)}.$$
The Schr\"odinger-Hartree equation \eqref{S} is said to be mass-critical or $L^2$-critical if $s_c=0$, which is equivalent to $p=1+\frac{\alpha+2+2b}N$. It is called energy-critical or $\dot H^1$-critical if $s_c=1$, which corresponds to $p=1+\frac{2+2b+\alpha}{N-2}$. This work is concerned with the mass super-critical and energy sub-critical regimes $0<s_c<1$.\\

To the author knowledge, the inhomogeneous Choquard problem, was only treated in \cite{saa}, where the local well-posedness in the energy space was proved via an adapted Gagliardo-Nirenberg type identity. In the focusing regime, a threshold of global existence versus finite time blow-up of solutions was obtained.\\

 On the other hand, the particular case $b=0$ in \eqref{S}, was investigated in many directions. Indeed, the existence of the energy sub-critical and mass-super-critical (and the mass-critical non-global) solutions to the non-linear Choquard problem were investigated in \cite{fy}. In the defocusing case, the scattering of energy global solutions was obtained in \cite{tv}. In the focusing sign, the existence and asymptotic properties of standing waves, were treated in \cite{mvs}. \\

It is the aim of this paper to extend the previous work \cite{saa}, by the study of the asymptotic behavior of the energy global solutions to the focusing Choquard problem \eqref{S}. Indeed, one proves that every global solution to the Choquard equation \eqref{S} is asymptotic, as $t\to\pm\infty,$ to a solution of the associated linear Schr\"odinger equation. \\

The method used in this note, due to \cite{dm}, is based on a scattering criterion \cite{tao} and Morawetz estimates. In \cite{st2} the scattering for the homogeneous non-linear Choquard problem, which corresponds to $b=0$ in \eqref{S}, was proved using the concentration-compactness-rigidity approach introduced in \cite{km}. This result was revisited recently in \cite{aka} by giving an alternative proof.\\% of the scattering to the focusing homogeneous Choquad equation.\\

The rest of this paper is organized as follows. The next section contains the main results and some technical tools needed in the sequel. Some variational estimates are given in section three. Section four is devoted to establish a scattering criterion. In the fifth section, a Morawetz estimate is proved. The last section is concerned with the proof of scattering. Finally, a variance identity and a Strichartz type estimate of the source term are proved in the Appendix.\\

Here and hereafter, $C$ will denote a constant which may vary from line to line and if $A$ and $B$ are non-negative real numbers, $A\lesssim B$  means that $A\leq CB$. \\
Denote, for simplicity, the Lebesgue space $L^r:=L^r({\R^N})$ with the usual norm $\|\cdot\|_r:=\|\cdot\|_{L^r}$ and $\|\cdot\|:=\|\cdot\|_2$. Take $H^1:=H^1({\R^N})$ be the usual inhomogeneous Sobolev space endowed with the complete norm 
$$\|\cdot\|_{H^1} := \Big(\|\cdot\|^2 + \|\nabla \cdot\|^2\Big)^\frac12.$$
If $X$ is an abstract space $C_T(X):=C([0,T],X)$ stands for the set of continuous functions valued in $X$ and $X_{rd}$ is the subset of radial elements. Moreover, for an eventual solution to \eqref{S}, $T^*>0$ denotes it's lifespan. Finally, $x^\pm$ are two real numbers near to $x$ satisfying $x^+>x$ and $x^-<x$.
%$$H=\underbrace{H^1({\R^N})\times...\times H^1({\R^N})}_{m~fois}.$$
%To simplify the notation, we write $\Psi =(\psi_{1,0},...,\psi_{m,0})$ as the initial data and $ u  = (u_1,...,u_m)$ as the solution.
%%%%%%%%%%%%%%%%%%%%%%%%%%%%%%%%%%%%%%%%%%%%%%%%%%%%%%%%%%%%%%%%%%%%%%%%%%%%%%%%%%%%%%%%%%%%%%%%%%%%%%%%%%%%
\section{Background and main results.}
%%%%%%%%%%%%%%%%%%%%%%%%%%%%%%%%%%%%%%%%%%%%%%%%%%%%%%%%%%%%%%%%%%%%%%%%%%%%%%%%%%%%%%%%%%%%%%%
This section contains the contribution of this paper and some standard estimates needed in the sequel.
%%%%%%%%%%%%%%%%%%%%%%%%%%%%%%%%%%%%%%%%%%%%%%%%%%%%%%%%%%%%%%%%%%%%%%%%%%%%%%%%%%%%%%%%%%%%%%%%%%%%%%%%%%%%
\subsection{Preliminary}
%%%%%%%%%%%%%%%%%%%%%%%%%%%%%%%%%%%%%%%%%%%%%%%%%%%%%%%%%%%%%%%%%%%%%%%%%%%%%%%%%%%%%%%%%%%%%%%
 Here and hereafter denote the following quantities. For $a,b\in\R$ and $u\in H^1$, let the real numbers
\begin{gather*}
p_*:=1+\frac{\alpha+2+2b}N,\quad p^*:=1+\frac{\alpha+2+2b}{N-2};\\
B:=Np-N-\alpha-2b,\quad A:=2p-B.
\end{gather*}
Denote the operator $\left\langle \cdot\right\rangle:=(1+\nabla)\cdot$ and the source term
$$\mathcal N:=\mathcal N(x,u):=(I_\alpha*|\cdot|^b|u|^p)|x|^b|u|^{p-2}u.$$
Take also a radial smooth function
$$\psi\in C_0^\infty(\R^N),\quad supp(\psi)\subset \{|x|<1\}, \quad\psi=1\,\,\mbox{on}\,\,\{|x|<\frac12\},\quad\psi_R:=\psi(\frac\cdot R).$$
The existence of ground states was established in \cite{saa}.
\begin{defi}
We call ground state of \eqref{S}, any solution to 
\begin{equation}\label{grnd}
\Delta\phi-\phi+(I_\alpha*|\cdot|^b|\phi|^p)|x|^b|\phi|^{p-2}\phi=0,\quad0\neq\phi\in H^1,
\end{equation}
which minimizes the problem
%\begin{equation}\label{min}
$$m:=\inf_{0\neq u\in{H^1}}\Big\{S(u):=(E+M)(u) \quad\mbox{s\,. t}\quad \mathcal I(u)=0\Big\},$$%\end{equation}
where
$$\mathcal I(u):=\frac 4N\Big(\|\nabla u\|^2-\frac B{4p}\int_{\R^N}(I_\alpha*|\cdot|^b|u|^p)|x|^b|u|^p\,dx\Big).$$
\end{defi}
%\begin{rem}The existence of ground states was established in \cite{saa}.\end{rem}
The next Gagliardo-Nirenberg type estimates related to the inhomogeneous Choquard problem \eqref{S} were proved in \cite{saa}.
\begin{prop}\label{gag}
Let $N\geq1$,  $b,\alpha$ satisfying \eqref{cnd} and $1+\frac{2b+\alpha}N< p< p^*$. Then, 
\begin{enumerate}
\item[1.]
there exists a positive constant $C(N,p,b,\alpha)$, such that for any $u\in H^1$,
%\begin{equation}\label{ineq}
$$\int_{\R^N}(I_\alpha*|\cdot|^b|u|^p)|x|^b|u|^p\,dx\leq C(N,p,b,\alpha)\|u\|^A\|\nabla u\|^B;$$
%\end{equation}
%Moreover, 
\item[2.]
%denoting $\beta:=\frac1{C(N,p,\gamma,\alpha)}$, it follows that
the next minimization problem is attained in some $Q\in H^1$,
$$\frac1{C(N,p,b,\alpha)}=\inf\Big\{J(u):=\frac{\|u\|^A\|\nabla u\|^B}{\int_{\R^N}(I_\alpha*|\cdot|^b|u|^p)|x|^b|u|^p\,dx},\quad0\neq u\in H^1\Big\}$$
such that ${C(N,p,b,\alpha)}=\int_{\R^N}(I_\alpha*|\cdot|^b|Q|^{p})|x|^b|Q|^p\,dx$ and
%\begin{equation}\label{euler}
$$-B\Delta Q+AQ-\frac{2p}{C(N,p,b,\alpha)}(I_\alpha*|\cdot|^b|Q|^p)|x|^b|Q|^{p-2}Q=0;$$
%\end{equation}
%and $\beta=(\int|x|^\gamma|\psi|^{p+1}\,dx)^{-1}$;
%\end{thm}\begin{thm}\label{t0''}
%Let $\gamma\geq0$, $\alpha\in(0,1)$ and $2<p-\frac{2\gamma}{N-2\alpha}<\frac{2N}{N-2\alpha}$. Then, %the best constant $C(N,p,\gamma,\alpha)$ in \eqref{ineq} equals
\item[3.]
furthermore, 
\begin{equation}\label{part3}
%\beta=\frac{A}{2p}(\frac AB)^{-\frac B4}\|\phi\|^{2p-1}.
C(N,p,b,\alpha)=\frac{2p}{A}(\frac AB)^{\frac{B}2}\|\phi\|^{-2(p-1)},
\end{equation}
where $\phi$ is a ground state solution to \eqref{grnd}.
%\frac{A}{1+p}(\frac AB)^{-\frac B2}\|\phi\|^{p-1}.$$
%where $\phi$ is a ground state solution to \eqref{grnd}.
%\begin{equation}\label{grnd} 
%(-\Delta)^s\phi +\phi -(I_\alpha*|\phi|^p)|\phi|^{p-2}\phi=0 ,\quad 0\neq \phi \in H^1.
%\end{equation}
\end{enumerate}
\end{prop}
Here and hereafter, $\phi$ denotes a solution to \eqref{grnd}, which satisfies \eqref{part3}. The next scale invariant quantities \cite{dr} called respectively, mass-energy and mass-gradient give a threshold of global existence and scattering versus finite time blow-up of energy solutions to the above Choquard problem.
$$\mathcal M\mathcal E [u]:=\frac{E[u]^{s_c}M[u_0]^{1-s_c}}{E[\phi]^{s_c}M[\phi]^{1-s_c}},\quad\mathcal G\mathcal M[u]:=\frac{\|\nabla u\|^{s_c}\|u\|^{1-s_c}}{\|\nabla\phi\|^{s_c}\|\phi\|^{1-s_c}}.$$
\begin{rem}
Using Pohozaev identities, the terms $E[\phi]^{s_c}M[\phi]^{1-s_c}$ and $\|\nabla\phi\|^{s_c}\|\phi\|^{1-s_c}$ are invariant of the choice of $\phi$.
\end{rem}
The inhomogeneous Choquard problem \eqref{S} is locally well-posed in the energy space \cite{saa}.
\begin{prop}\label{t0}
Let $N\geq3$, $b,\alpha$ satisfying \eqref{cnd}, $u_0\in H^1$ and $2\leq p< p^*$. Then, there exists $T>0$ and a unique local solution to the inhomogeneous Choquard problem \eqref{S},
$$ u\in C([0,T],H^1).$$ 
Moreover,
\begin{enumerate}
\item[1.]  the solution satisfies the mass and energy conservation laws
\begin{gather*}
Mass:=M(u(t)) :=\int_{\R^N}|u(t,x)|^2dx = M(u_0);\\
Energy:=E(u(t)) :=\|\nabla u(t)\|^2-\frac1p\int_{\R^N}(I_\alpha *|\cdot|^b|u(t)|^p)|x|^b|u(t,x)|^p\,dx= E(u_0); 
\end{gather*}
\item[2.] $u\in L^q_{loc}((0,T),W^{1,r})$ for any Strchartz admissible pair $(q,r)$.% in the meaning of the next definition.
\end{enumerate}
\end{prop}
Let us end this sub-section with some definitions in the spirit of \cite{dm}. Take, for $R>>1$, the radial function defined on $\R^N$ by 
$$a(x):=\left\{
\begin{array}{ll}
\frac12|x|^2,\quad |x|\leq \frac R2 ;\\
R|x|,\quad |x|>R.
%\label{S}
\end{array}
\right.$$
and satisfying in the centered annulus $C(0,\frac R2,R)$,
$$\partial_r a > 0,\quad\partial^2_r a \geq 0,\quad |\partial^\alpha a|\leq C_\alpha R|\cdot|^{1-\alpha}\quad\forall\, |\alpha| \geq 1.$$
Here, $\partial_r a:= \frac\cdot{|\cdot|}.\nabla a$ denotes the radial derivative. Note that on the centered ball of radius $\frac R2$, one has
$$a_{jk}=2\delta_{jk},\quad \Delta a = 2N\quad\mbox{and}\quad \Delta^2 a = 0.$$
Moreover, for $|x| > R$, 
$$a_{jk} =\frac R{|x|}\Big(\delta_{jk}-\frac{x_jx_k}{|x|^2}\Big)\quad\mbox{and}\quad \Delta a=\frac{(N -1)R}{|x|},\,\, \Delta^2 a = 0.$$
Define the variance potential
$$V_a:=\int_{\R^N}a(x)|u(.,x)|^2\,dx,$$
and the Morawetz action
$$M_a:=2\Im\int_{\R^N} (a_ju_j)\bar u\,dx=2\Im\int_{\R^N} (\nabla a\nabla u)\bar u\,dx.$$
With a standard way, subscripts denote partial derivatives and repeated index are summed.
%%%%%%%%%%%%%%%%%%%%%%%%%%%%%%%%%%%%%%%%%%%%%%%%%%%%%%%%%%%%%%%%%%%%%%%%%%%%%%%%%%%%%%%%%%%%%%%%%%%%%%%%%%%%
\subsection{Main result}
%%%%%%%%%%%%%%%%%%%%%%%%%%%%%%%%%%%%%%%%%%%%%%%%%%%%%%%%%%%%%%%%%%%%%%%%%%%%%%%%%%%%%%%%%%%%%%%
The main goal of this manuscript is to prove the following scattering result.
\begin{thm}\label{sctr}
Take $N\geq3$, $b,\alpha$ satisfying \eqref{cnd}, $p_*<p<p^*$ such that $p\geq2$ and $u_0\in H^1_{rd}$ satisfying 
$$\max\{\mathcal M\mathcal E(u_0),\mathcal G\mathcal M(u_0)\}<1.$$
Then, there exist a unique global solution $u\in C(\R,H^1)$ to \eqref{S}, which scatters in both directions. Precisely, there exists $u_\pm\in H^1$ such that
$$\lim_{t\to\pm\infty}\|u(t)-e^{it\Delta}u_\pm\|_{H^1}=0.$$
\end{thm}
\begin{rems}
\begin{enumerate}
\item[1.]
The global existence of energy solutions to \eqref{S} was proved in \cite{saa};
\item[2.]
the radial assumption allows to use the following Strauss inequality, which is available for any $u\in H^1(\R^N)$ such that $N\geq2$,
$$|\cdot|^{\frac{N-1}2}|u|\leq C_N\|u\|_{H^1},\quad\mbox{almost everywhere};$$
\item[3.] the restriction $N\geq3$ is needed in the proof of the scattering criterion, precisely when using the dispersion of the free Schro\"odinger kernel.
\end{enumerate}
\end{rems}
In order to prove the scattering, one needs the following scattering criterion \cite{tao}.
\begin{prop}\label{crt}
Take $N\geq3$, $b,\alpha$ satisfying \eqref{cnd}, $p_*<p<p^*$ such that $p\geq2$ and $u\in C(\R,H^1_{rd})$ be a global radial solution to \eqref{S}. Assume that 
$$0<\sup_{t\geq0}\|u(t)\|_{H^1}:=E<\infty.$$
There exist $R,\epsilon>0$ depending on $E,N,p,b,\alpha$ such that if
%\begin{equation}\label{crtr} 
$$\liminf_{t\to+\infty}\int_{|x|<R}|u(t,x)|^2\,dx<\epsilon,$$
%\end{equation}
then, $u$ scatters for positive time. %Precisely, there exists $u_+\in H^1$ such that
%$$\lim_{t\to+\infty}\|u(t)-e^{it\Delta}u_+\|_{H^1}=0.$$
\end{prop}
The following Morawetz estimate stand for a standard tool to prove scattering.
\begin{prop}\label{cr}
Take $N\geq3$, $b,\alpha$ satisfying \eqref{cnd}, $p_*<p<p^*$ such that $p\geq2$ and $u\in C(\R,H^1_{rd})$ be a global solution to \eqref{S}. Then, for $T>0$, $\delta>0$ (given by Lemma \ref{bnd}) and $R:=R(\delta,M(u),\phi)$ large enough, holds 
$$\frac1T\int_0^T\Big(\int_{|x|<R}|u(t,x)|^{\frac{2NP}{\alpha+N}}\Big)^{\frac{\alpha+N}{Np}}\,dt\lesssim \frac RT+R^{-\frac{(N-1)B}N}.$$
\end{prop}
%%%%%%%%%%%%%%%%%%%%%%%%%%%%%%%%%%%%%%%%%%%%%%%%%%%%%%%%%%%%%%%%%%%%%%%%%%%%%%%%%%%%%%%%%%%%%%%%%%%%%%%%%%%%
\subsection{Useful estimates}
%%%%%%%%%%%%%%%%%%%%%%%%%%%%%%%%%%%%%%%%%%%%%%%%%%%%%%%%%%%%%%%%%%%%%%%%%%%%%%%%%%%%%%%%%%%%%%%
Recall the so-called Strichartz estimate \cite{gw}.
\begin{defi}
Take $N\geq3$ and $s\in[0,1)$. A couple of real numbers $(q,r)$ is said to be $s$-admissible if 
$$\frac{2N}{N-2s}\leq r<\frac{2N}{N-2},\quad2\leq q,r\leq\infty\quad\mbox{and}\quad N(\frac12-\frac1r)=\frac{2}q+s.$$
Denote the set of $s$-admissible pairs by $\Gamma_s$, $\Gamma:=\Gamma_0$ and $(q,r)\in\Gamma'_s$ if $(q',r')\in\Gamma_s$.
\end{defi}
\begin{prop}\label{prop2}
Let $N \geq3$, $s\in[0,1)$, $b,\alpha$ satisfying \eqref{cnd}, $p_*<p<p^*$ and $u_0\in\dot H^s$. Then,
%\begin{enumerate}\item[1.]
$$\sup_{(q,r)\in\Gamma_s}\|u\|_{L^q(L^r)}\lesssim\|u_0\|_{\dot H^s}+\inf_{(q',r')\in\Gamma'_{-s}}\|i\dot u+\Delta  u\|_{L^{ q'}(L^{ r'})};$$
%\item[2.]\end{enumerate}
\end{prop}
\begin{rem}\label{rm}
Take the inhomogeneous Schr\"odinger equation $i\dot u+\Delta u=h=h\chi_{|x|<1}+h\chi_{|x|\geq1}:=h_1+h_2$. Writing $u:=u_1+u_2$, where $i\dot u_i+\Delta u_i=h_i$, one gets the Strichartz estimate
$$\sup_{(q,r)\in\Gamma_s}\|u\|_{L^q_t(L^r)}\lesssim\|u_0\|_{\dot H^s}+\inf_{(q',r')\in\Gamma'_{-s}}\|h\|_{L^{ q'}_t(L^{ r'}(|x|<1))}+\inf_{(q',r')\in\Gamma'_{-s}}\|h\|_{L^{ q'}_t(L^{ r'}(|x|>1))}.$$
\end{rem}
The next variance identity is proved in the appendix.
\begin{prop}\label{vrnc}
Take $N\geq3$, $b,\alpha$ satisfying \eqref{cnd}, $p_*<p<p^*$ such that $p\geq2$ and $u\in C_T(H^1)$ be a local solution to \eqref{S}. Then, holds on $[0,T]$,
\begin{eqnarray*}
V''_a
&=&M_a'\\
&=&4\int_{\R^N}\partial_l\partial_ka\Re(\partial_ku\partial_l\bar u)\,dx-\int_{\R^N}\Delta^2a|u|^2\,dx\\
&+&2(\frac2p-1)\int_{\R^N}\Delta a(I_\alpha*|\cdot|^b|u|^{p})|x|^b|u|^p\,dx\\
&+&\frac{4b}p\int_{\R^N}x\nabla a|x|^{b-2}(I_\alpha*|\cdot|^b|u|^p)|u|^{p}\,dx\\
&+&\frac{4}p(\alpha-N)\int_{\R^N}\nabla a(\frac.{|\cdot|^2}I_\alpha*|\cdot|^b|u|^p)|x|^b|u|^{p}\,dx.
\end{eqnarray*}
%$$M_a'(t)=.$$
\end{prop}
Now, one gives some standard estimates independent of the Schr\"odinger equation \eqref{S}. Recall a Hardy-Littlewood-Sobolev inequality \cite{el}.
\begin{lem}\label{hls}
Let $0 <\lambda < N\geq1$ and $1<s,r<\infty$ be such that $\frac1r +\frac1s +\frac\lambda N = 2$. Then,
$$\int_{\R^N\times\R^N} \frac{f(x)g(y)}{|x-y|^\lambda}\,dx\,dy\leq C(N,s,\lambda)\|f\|_{r}\|g\|_{s},\quad\forall f\in L^r(\R^N),\,\forall g\in L^s(\R^N).$$
\end{lem}
Let us give a useful consequence \cite{st}.
\begin{cor}\label{cor}\label{lhs2}
Let $0 <\lambda < N\geq1$ and $1<s,r,q<\infty$ be such that $\frac1q+\frac1r+\frac1s=1+\frac\alpha N$. Then,%Assume that $f\in L^s({\R^N})$ and $g\in L^q({\R^N})$. Then,
$$\|(I_\alpha*f)g\|_{r'}\leq C(N,s,\alpha)\|f\|_{s}\|g\|_{q},\quad\forall f\in L^s({\R^N}), \,\forall g\in L^q({\R^N}).$$
\end{cor}
%Denoting $\frac{2N}{N-2}=\infty$ if $N\leq2$, 
Recall some Sobolev injections \cite{co}.
\begin{lem}\label{sblv}
Let $N\geq1$, then 
\begin{enumerate}
\item[1.]
$H^1\hookrightarrow L^q$ for any $q\in[2,\frac{2N}{N-2}]$ if $N\geq3$ and for any $q\in[2,\infty)$ if $N\leq2$;
\item[2.]
%$H^1_{rd} \hookrightarrow\hookrightarrow L^q$ is compact for any $N\geq2$ and $q\in(2,\frac{2N}{N-2})$, where $\frac{2N}{N-2}=\infty$, for $N=2$;
%\item[3.]
%$\{u\in H^1_{rd}(\R)$, s.t $u$ is decreasing $\} \hookrightarrow\hookrightarrow L^q(\R)$ for any $q\in(2,\infty)$;
%\item[4.]
if $N\geq3$ and $2\leq p\leq\frac{2N}{N-2}$ or $N\leq2$ and $2\leq q<\infty$, so
%\begin{equation}\label{gnn}
$$\|u\|_{{p}}\lesssim\|u\|^{1-N(\frac12-\frac1p)}\|\nabla u\|^{N(\frac12-\frac1p)},\quad \forall u\in H^1(\R^N).$$
%\end{equation}for any $u\in H^1(\R^N)$.
\end{enumerate}
\end{lem}
The following interpolation estimate will be useful \cite{nir}.
\begin{prop}
Let $1\leq q,r\leq\infty$ and $0\leq s<m$. Then,
$$\|(-\Delta)^{\frac s2}\cdot\|_p\lesssim \|(-\Delta)^{\frac m2}\cdot\|_r^\theta\|\cdot\|_q^{1-\theta},$$
for any $\theta\in[\frac ms,1]$ satisfying
$$\frac1p=\frac sN+\theta(\frac12-\frac1N)+\frac{1-\theta}q.$$
\end{prop}
The next abstract result ends this section.% \cite{taobk}.
\begin{lem}\label{abs}%{{(Bootstrap Lemma)}}
Let $T>0$ and $X\in C([0,T],\R_+)$ such that $$X\leq a+bX^{\theta}\mbox{ on } [0,T],$$
where $a$, $b>0$, $\theta>1$, $a<(1-\frac{1}{\theta})(\theta b)^{\frac{1}{1-\theta}}$ and $X(0)\leq (\theta b)^{\frac{1}{1-\theta}}$. Then
$$X\leq\frac{\theta}{\theta -1}a \mbox{ on } [0,T].$$
\end{lem}
\begin{proof}
The function $f(x):=bx^\theta-x +a$ is decreasing on $[0,(b\theta)^{\frac1{1-\theta}}]$ and increasing on $[(b\theta)^\frac1{1-\theta} ,\infty)$. The assumptions imply that $f((b\theta)^\frac1{1-\theta})< 0$ and $f(\frac\theta{\theta-1}a)\leq0$. As $f(X(t))\geq 0$, $f(0) > 0$ and $X(0)\leq(b\theta)^\frac1{1-\theta}$, we conclude the proof by a continuity argument.
\end{proof}
%%%%%%%%%%%%%%%%%%%%%%%%%%%%%%%%%%%%%%%%%%%%%%%%%%%%%%%%%%%%%%%%%%%%%%%%%%%
\section{Variational Analysis}
%%%%%%%%%%%%%%%%%%%%%%%%%%%%%%%%%%%%%%%%%%%%%%%%%%%%%%%%%%%%%%%%%%%%%%%%%%%%%%%%%%%%%%%%%%%%%%%%%%%%%%
In this section, one collects some estimates needed in the proof of the main result.
\begin{lem}\label{bnd}
Take $N\geq3$, $b,\alpha$ satisfying \eqref{cnd}, $p_*<p<p^*$ such that $p\geq2$ and $u_0\in H^1$ satisfying 
$$\max\{\mathcal M\mathcal E(u_0),\mathcal G\mathcal M(u_0)\}<1.$$
Then, there exists $\delta>0$ such that the solution $u\in C(\R,H^1)$ satisfies
$$\max\{\sup_{t\in\R}\mathcal M\mathcal E(u(t)),\sup_{t\in\R}\mathcal G\mathcal M(u(t))\}<1-\delta.$$
\end{lem}
\begin{proof}
Denote $C_{N,p,b,\alpha}:=C(N,p,b,\alpha)$ given by Proposition \ref{gag}. The inequality $\mathcal M\mathcal E(u_0)<1$ gives the existence of $\delta>0$ such that 
\begin{eqnarray*}
1-\delta
&>&\frac{M(u_0)^{\frac{1-s_c}{s_c}}E(u_0)}{M(\phi)^{\frac{1-s_c}{s_c}}E(\phi)}\\
&>&\frac{M(u_0)^{\frac{1-s_c}{s_c}}}{M(\phi)^{\frac{1-s_c}{s_c}}E(\phi)}\Big(\|\nabla u(t)\|^2-\frac1p\int_{\R^N}(I_\alpha*|\cdot|^b|u|^p)|x|^b|u|^p\,dx\Big)\\
&>&\frac{M(u_0)^{\frac{1-s_c}{s_c}}}{M(\phi)^{\frac{1-s_c}{s_c}}E(\phi)}\Big(\|\nabla u(t)\|^2-\frac{C_{N,p,b,\alpha}}p\|u\|^A\|\nabla u(t)\|^B\Big).
\end{eqnarray*}
Thanks to Pohozaev identities, one has
$$E(\phi)=\frac{B-2}B\|\nabla\phi\|^2=\frac{B-2}A\|\phi\|^2.$$
Thus,
\begin{eqnarray*}
1-\delta
&>&\frac B{B-2}\frac{M(u_0)^{\frac{1-s_c}{s_c}}}{M(\phi)^{\frac{1-s_c}{s_c}}\|\nabla\phi\|^2}\Big(\|\nabla u(t)\|^2-\frac{C_{N,p,b,\alpha}}p\|u\|^A\|\nabla u(t)\|^B\Big)\\
&>&\frac B{B-2}\frac{M(u_0)^{\frac{1-s_c}{s_c}}\|\nabla u(t)\|^2}{M(\phi)^{\frac{1-s_c}{s_c}}\|\nabla\phi\|^2}-\frac B{B-2}\frac{M(u_0)^{\frac{1-s_c}{s_c}}}{M(\phi)^{\frac{1-s_c}{s_c}}\|\nabla\phi\|^2}\frac{C_{N,p,b,\alpha}}p\|u\|^A\|\nabla u(t)\|^B\\
&>&\frac B{B-2}\frac{M(u_0)^{\frac{1-s_c}{s_c}}\|\nabla u(t)\|^2}{M(\phi)^{\frac{1-s_c}{s_c}}\|\nabla\phi\|^2}-\frac B{B-2}\frac{M(u_0)^{\frac{1-s_c}{s_c}}}{M(\phi)^{\frac{1-s_c}{s_c}}\|\nabla\phi\|^2}\frac2A(\frac AB)^\frac{B}2\|\phi\|^{-2(p-1)}\|u\|^A\|\nabla u(t)\|^B.
\end{eqnarray*}
So,
\begin{eqnarray*}
1-\delta
&>&\frac B{B-2}\frac{M(u_0)^{\frac{1-s_c}{s_c}}\|\nabla u(t)\|^2}{M(\phi)^{\frac{1-s_c}{s_c}}\|\nabla\phi\|^2}-\frac B{B-2}\frac2A\frac{M(u_0)^{\frac{1-s_c}{s_c}}}{M(\phi)^{\frac{1-s_c}{s_c}}\|\nabla\phi\|^2}(\frac{\|\phi\|}{\|\nabla\phi\|})^{B}\|\phi\|^{-2(p-1)}\|u\|^A\|\nabla u(t)\|^B\\
&>&\frac B{B-2}\frac{M(u_0)^{\frac{1-s_c}{s_c}}\|\nabla u(t)\|^2}{M(\phi)^{\frac{1-s_c}{s_c}}\|\nabla\phi\|^2}-\frac B{B-2}\frac2A\frac{\|u_0\|^{A+2\frac{1-s_c}{s_c}}}{M(\phi)^{\frac{1-s_c}{s_c}}\|\nabla\phi\|^2}(\frac{\|\phi\|}{\|\nabla\phi\|})^{B}\|\phi\|^{-2(p-1)}\|\nabla u(t)\|^B.
\end{eqnarray*}
Using the equalities $s_c=\frac{B-2}{2(p-1)}$ and $\frac BA=(\frac{\|\nabla\phi\|}{\|\phi\|})^2$, one has
\begin{eqnarray*}
1-\delta
&>&\frac B{B-2}\frac{M(u_0)^{\frac{1-s_c}{s_c}}\|\nabla u(t)\|^2}{M(\phi)^{\frac{1-s_c}{s_c}}\|\nabla\phi\|^2}-\frac B{B-2}\frac2A\frac{(\|u_0\|^{\frac{1-s_c}{s_c}}\|\nabla u(t)\|)^B}{M(\phi)^{\frac{1-s_c}{s_c}}\|\nabla\phi\|^2}(\frac{\|\phi\|}{\|\nabla\phi\|})^{B}\|\phi\|^{-2(p-1)}\\
&>&\frac B{B-2}\frac{M(u_0)^{\frac{1-s_c}{s_c}}\|\nabla u(t)\|^2}{M(\phi)^{\frac{1-s_c}{s_c}}\|\nabla\phi\|^2}-\frac2{B-2}\frac{(\|u_0\|^{\frac{1-s_c}{s_c}}\|\nabla u(t)\|)^B}{\|\phi\|^{2\frac{1-s_c}{s_c}-B+2p}\|\nabla\phi\|^{B}}\\
&>&\frac B{B-2}\Big(\frac{\|u_0\|^{\frac{1-s_c}{s_c}}\|\nabla u(t)\|}{\|\phi\|^{\frac{1-s_c}{s_c}}\|\nabla\phi\|}\Big)^2-\frac2{B-2}\Big(\frac{\|u_0\|^{\frac{1-s_c}{s_c}}\|\nabla u(t)\|}{\|\phi\|^{\frac{1-s_c}{s_c}}\|\nabla\phi\|}\Big)^B.
\end{eqnarray*}
Take the real function defined on $[0,1]$ by $f(x):=\frac B{B-2}x^2-\frac2{B-2}x^B$, with first derivative $f'(x)=\frac{2B}{B-2}x(1-x^{B-2})$. Thus, with the table change of $f$ and the continuity of 
$t\to X(t):=\frac{\|u_0\|^{\frac{1-s_c}{s_c}}\|\nabla u(t)\|}{\|\phi\|^{\frac{1-s_c}{s_c}}\|\nabla\phi\|}$, it follows that $X(t)<1$ for any $t<T^*$. Thus, $T^*=\infty$ and there exists $\epsilon>0$ near to zero such that $X(t)\in f^{-1}([0,1-\delta])=[0,1-\epsilon]$. This finishes the proof.
\end{proof}
Let us prove the coercivity estimate on balls of large radials.
\begin{lem}\label{crcv}
There exists $R_0:=R_0(\delta,M(u),\phi)>0$ such that for any $R>R_0$,
$$\sup_{t\in\R}\|\psi_R u(t)\|^{1-s_c}\|\nabla(\psi_Ru(t))\|^{s_c}<(1-\delta)\|\phi\|^{1-s_c}\|\nabla\phi\|^{s_c}.$$
In particular, there exists $\delta'>0$ such that
$$\|\nabla(\psi_Ru)\|^2-\frac B{2p}\int_{\R^N}(I_\alpha*|\cdot|^b|\psi_Ru|^p)|x|^b|\psi_Ru|^p\,dx\geq\delta'\|\psi_Ru\|_{\frac{2Np}{N+\alpha+2b}}^2.$$ 
\end{lem}
\begin{proof}
Taking account of \eqref{gag}, one gets
\begin{eqnarray*}
E(u)
&=&\|\nabla u\|^2-\frac1p\int_{\R^N}(I_\alpha*|\cdot|^b|u|^p)|x|^b|u|^p\,dx\\
&\geq&\|\nabla u\|^2\Big(1-\frac{C_{N,p,b,\alpha}}p\|u\|^A\|\nabla u\|^{B-2}\Big)\\
%&\geq&\|\nabla u\|^2\Big(1-\frac{C_{N,p,b,\alpha}}p\|u\|^{2(p-1)(1-s_c)}\|\nabla u\|^{2(p-1)s_c}\Big)\\
&\geq&\|\nabla u\|^2\Big(1-\frac{C_{N,p,b,\alpha}}p[\|u\|^{1-s_c}\|\nabla u\|^{s_c}]^{2(p-1)}\Big).
\end{eqnarray*}
So,
\begin{eqnarray*}
E(u)
&\geq&\|\nabla u\|^2\Big(1-(1-\delta)\frac2A(\frac AB)^{\frac B2}\|\phi\|^{-2(p-1)}[\|\phi\|^{1-s_c}\|\nabla\phi\|^{s_c}]^{2(p-1)}\Big)\\
&\geq&\|\nabla u\|^2\Big(1-(1-\delta)\frac2A(\frac AB)^{\frac B2}[\frac{\|\nabla\phi\|}{\|\phi\|}]^{2s_c(p-1)}\Big)\\
&\geq&\|\nabla u\|^2\Big(1-(1-\delta)\frac2B(\frac{\|\phi\|}{\|\nabla\phi\|})^{B-2}[\frac{\|\nabla\phi\|}{\|\phi\|}]^{B-2}\Big)\\
&\geq&\|\nabla u\|^2\Big(1-(1-\delta)\frac2B\Big).
\end{eqnarray*}
Thus, using Sobolev injections with the fact that $p<p^*$, one gets
$$\|\nabla u\|^2-\frac B{2p}\int_{\R^N}(I_\alpha*|\cdot|^b|u|^p)|x|^b|u|^p\,dx\geq\delta\|\nabla u\|^2\geq\delta'\|u\|_{\frac{2Np}{N+\alpha+2b}}^2.$$ 
This gives the second part of the claimed Lemma provided that the first point is proved. Compute
\begin{eqnarray*}
\|\nabla(\psi_R u)\|^2
&=&\int_{\R^N}\Big(\psi_R|\nabla u|^2+|\nabla\psi_R|^2|u|^2+2\Re(\bar u\nabla u)\psi_R\nabla\psi_R)\Big)\,dx\\
&=&\int_{\R^N}\Big(\psi_R|\nabla u|^2+|\nabla\psi_R|^2|u|^2+\frac12\nabla(|u|^2)\nabla(|\psi_R|^2)\Big)\,dx\\
&=&\int_{\R^N}\Big(\psi_R|\nabla u|^2+|\nabla\psi_R|^2|u|^2-\frac12|u|^2\Delta(|\psi_R|^2)\Big)\,dx\\
&=&\int_{\R^N}\Big(\psi_R|\nabla u|^2+|\nabla\psi_R|^2|u|^2-|u|^2(\psi_R\Delta\psi_R+|\nabla\psi_R|^2)\Big)\,dx\\
&=&\int_{\R^N}\Big(\psi_R|\nabla u|^2-|u|^2\psi_R\Delta\psi_R\Big)\,dx\\
&\leq&\|\nabla u\|^2+C\frac{|u|^2}{R^2}.
\end{eqnarray*}
Then, one gets the proof of the first point and so the Lemma.
\end{proof}
%%%%%%%%%%%%%%%%%%%%%%%%%%%%%%%%%%%%%%%%%%%%%%%%%%%%%%%%%%%%%%%%%%%%%%%%%%%
%\section{Small data Theory}
%%%%%%%%%%%%%%%%%%%%%%%%%%%%%%%%%%%%%%%%%%%%%%%%%%%%%%%%%%%%%%%%%%%%%%%%%%%%%%%%%%%%%%%%%%%%%%%%%%%%%%
%\begin{prop}\label{sml}\end{prop}
%%%%%%%%%%%%%%%%%%%%%%%%%%%%%%%%%%%%%%%%%%%%%%%%%%%%%%%%%%%%%%%%%%%%%%%%%%%
\section{Scattering Criterion}
%%%%%%%%%%%%%%%%%%%%%%%%%%%%%%%%%%%%%%%%%%%%%%%%%%%%%%%%%%%%%%%%%%%%%%%%%%%%%%%%%%%%%%%%%%%%%%%%%%%%%%
This section is devoted to prove Proposition \ref{crt}. By Lemma \ref{bnd}, $u$ is bounded in $H^1$. Take $\epsilon>0$ near to zero and $R(\epsilon)>>1$ to be fixed later. Define the  Strichartz norm 
$$\|\cdot\|_{S^{s}(I)}:=\sup_{(q,r)\in\Gamma_s}\|\cdot\|_{L^q(I,L^r)}.$$%,\quad \mathcal N:=\mathcal N(x,u):=(I_\alpha*|\cdot|^b|u|^p)|x|^b|u|^pu.$$
Let us give a technical result.
\begin{lem}\label{tch}
Let $N\geq3$, $I$ a time slab, $b,\alpha$ satisfying \eqref{cnd} and $p_*<p<p^*$. Then, there exists $\theta\in(0,2p-1)$ such that
$$\|u-e^{i.\Delta}u_0\|_{S^{s_c}(I)}\lesssim \|u\|_{L^\infty(I,\dot H^1)}^\theta\|u\|^{2p-1-\theta}_{S^{s_c}(I)}.$$
\end{lem}
\begin{proof}
%It is sufficient to check the possibility of the assumptions on the couples. Indeed, write the explicit formulas
Take the real numbers
\begin{gather*}
%\mu:=(\frac N{|b|})^-,
\quad a:=\frac{2p-\theta}{1-s_c},\quad d:=\frac{2p-\theta}{1+(2p-1-\theta)s_c};\\
%q:=\frac{2(2p-\theta)}{2+s_c(2p-2-\theta)},\quad 
r:=\frac{2N(2p-\theta)}{(N-2s_c)(2p-\theta)-4(1-s_c)}.
\end{gather*}
The condition $\theta=0^+$ gives via some computations that the previous pairs satisfy the admissibility conditions. Indeed, $a>\frac 2{1-s_c}$ implies that $r<\frac{2N}{N-2}$. Moreover,
%$r>\frac{2N}{N-2s_c}$ because
$$r=\frac{2N(2p-\theta)}{(N-2s_c)(2p-\theta)-4(1-s_c)}>\frac{2N}{N-2s_c}.$$
Thus,
%$$ (q,r)\in \Gamma,\quad (a,r)\in \Gamma_{s_c},\quad (b,r)\in \Gamma_{-s_c},$$
$$ (a,r)\in \Gamma_{s_c},\quad (d,r)\in \Gamma_{-s_c}\quad\mbox{and}\quad (2p-1-\theta)d'=a.$$
Take two real numbers $r_1,\mu$ satisfying
$$1+\frac\alpha N=\frac2\mu+\frac{\theta}{r_1}+\frac{2p-\theta}r.$$
This gives
\begin{eqnarray*}
\frac{2N}\mu
&=&\alpha+N-\frac{N\theta}{r_1}-\frac{N(2p-\theta)}r\\
&=&\alpha+N-\frac{N\theta}{r_1}-\frac{(N-2s_c)(2p-\theta)-4(1-s_c)}2\\
&=&-2b-\frac{N\theta}{r_1}+\theta\frac{2+2b+\alpha}{2(p-1)}.
\end{eqnarray*}
Using  Hardy-Littlewood-Sobolev and H\"older estimates, one has
\begin{eqnarray*}
\|\mathcal N\|_{L^{b'}(I,L^{r'}(|x|<1))}
&\lesssim&\||x|^b\|_{L^\mu(|x|<1)}^2\|u\|_{L^\infty(I,L^{r_1})}^\theta\|\|u\|_r^{2p-1-\theta}\|_{L^{d'}(I)}.
\end{eqnarray*}
Then, choosing $r_1:=\frac{2N}{N-2}$, one gets because $p<p^*$,
$$\frac{2N}\mu+2b=\frac\theta{2(p-1)}(2+2b+\alpha-(p-1)(N-2))>0.$$
So, $|x|^b\in L^\mu(|x|<1)$. Then, by Sobolev injections, one gets
\begin{eqnarray*}
\|\mathcal N\|_{L^{b'}(I,L^{r'}(|x|<1))}
&\lesssim&\||x|^b\|_{L^\mu(|x|<1)}^2\|u\|_{L^\infty(I,L^{r_1})}^\theta\|\|u\|_r^{2p-1-\theta}\|_{L^{d'}(I)}\\
&\lesssim&\|u\|_{L^\infty(I,\dot H^1)}^\theta\|u\|_{L^a(I,L^r)}^{2p-1-\theta}.
\end{eqnarray*}
 Moreover, on the complementary of the unit ball, one takes $r_1:=2$, so because $p>p_*$, one has
$$\frac{2N}\mu+2b=\frac\theta{2(p-1)}(2+2b+\alpha-N(p-1))<0.$$
So, $|x|^b\in L^\mu(|x|>1)$. The proof is achieved via Strichartz estimates.
%Using  Hardy-Littlewood-Sobolev and H\"older estimates via Sobolev injections, one has
%\begin{eqnarray*}
%\|\mathcal N\|_{L^{b'}(I,L^{r'}(|x|<1))}
%&\lesssim&\||x|^b\|_{L^\mu(|x|<1)}^2\|u\|_{L^\infty(I,L^{\frac{2N}{N-2s_c}})}^\theta\|\|u\|_r^{2p-1-\theta}\|_{L^{b'}(I)}\\
%&\lesssim&\|u\|_{L^\infty(I,\dot H^{s_c})}^\theta\|u\|_{L^a(I,L^r)}^{2p-1-\theta}.
%\end{eqnarray*}
%Treating the integrals on $\{|x|>1\}$ similarly, the proof follows with Strichartz estimates. 
\end{proof}
The following estimate will be proved in the appendix.
\begin{lem}\label{tch2}
Take $N\geq3$, $b,\alpha$ satisfying \eqref{cnd}, $p_*<p<p^*$ such that $p\geq2$ and $u\in C(\R,H^1)$ be a global solution to \eqref{S}. Then, there exist $2<q_1,q_2<\frac{2N}{N-2}$ and $0<\theta_1,\theta_2<2p$ such that
$$\|\left\langle u-e^{i.\Delta}u_0\right\rangle\|_{S(0,T)}
\lesssim\|u\|_{L^\infty_T(L^{q_1})}^{\theta_1}\|\left\langle u\right\rangle\|_{S(0,T)}^{2p-\theta_1}+\|u\|_{L^\infty_T(L^{q_2})}^{\theta_2}\|\left\langle u\right\rangle\|_{S(0,T)}^{2p-\theta_2}.$$
\end{lem}
 Using Strichartz estimates, write
$$\|e^{i.\Delta}u_0\|_{S^{s_c}(\R)}\lesssim1.$$
Taking account of the hypothesis, there exist $T>0$ such that
$$\max\{\|e^{i.\Delta}u_0\|_{S^{s_c}(T,\infty)},\|u(T)\|_{L^2(|x|<R)}\}<\epsilon.$$
Thanks to the equation \eqref{S}, if $\frac1R\lesssim \epsilon^{1+\beta}$ for some $\beta>0$ near to zero, one has
$$|\partial_t\int_{\R^N}\psi_R(x)|u(t,x)|^2\,dx|\lesssim\frac1R,\quad\sup_{t\in[T-\epsilon^{-\beta},T]}\int_{\R^N}\psi_R(x)|u(t,x)|^2\,dx\lesssim\epsilon.$$
%Denoting $\mathcal N:=(I_\alpha*|\cdot|^b|u|^p)|x|^b|u|^p$, 
Denote $J_2:=[0,T-\epsilon^{-\beta}]$ and $J_1:=[T-\epsilon^{-\beta},T]$. The integral formula gives
\begin{gather*}
u(T)=e^{iT\Delta}u_0+i\int_0^Te^{i(T-s)\Delta}\mathcal N\,ds;\\
e^{i(\cdot-T)}u(T)=e^{i\cdot\Delta}u_0+i\int_{J_1}e^{i(\cdot-s)\Delta}\mathcal N\,ds+i\int_{J_2}e^{i(\cdot-s)\Delta}\mathcal N\,ds.
\end{gather*}
 Strichartz estimates via Lemma \ref{tch} give
\begin{eqnarray*}
\|\int_{J_j}e^{i(\cdot-s)\Delta}\mathcal N\,ds\|_{S^{s_c}(T,\infty)}
&\lesssim&\|\mathcal N\|_{S'^{-s_c}(J_i)}\\
&\lesssim&\|u\|_{L^\infty(I_j,H^{s_c})}^\theta\|u\|^{2p-1-\theta}_{S^{s_c}(I_j)}\\
&\lesssim&\|u\|^{2p-1-\theta}_{S^{s_c}(I_j)}.
\end{eqnarray*}
On the other hand, if $(q,r)\in \Gamma_{s_c}$, write for $s^*:=\frac{2N}{N-2s}$,
\begin{eqnarray*}
\|u\|_{L^\infty(J_1,L^r)}
&\leq&\|\psi_Ru\|_{L^\infty(J_1,L^r)}+\|(1-\psi_R)u\|_{L^\infty(J_1,L^r)}\\
&\leq&\|\psi_Ru\|_{L^\infty(J_1,L^2)}^{\frac{2(s^*-r)}{r(s^*-2)}}\|\psi_Ru\|_{L^\infty(J_1,L^{s^*})}^{1-\frac{2(s^*-r)}{r(s^*-2)}}+\|(1-\psi_R)u\|_{L^\infty(J_1,L^\infty)}^{1-\frac2r}\|(1-\psi_R)u\|_{L^\infty(J_1,L^2)}^{\frac2r}\\
&\lesssim&\epsilon^{\frac{2(s^*-r)}{r(s^*-2)}}+(R^{-\frac{N-1}2}\|u\|_{L^\infty(J_1,H^1)})^{1-\frac2r}\|u_0\|^{\frac2r}\\
&\lesssim&\epsilon^{\frac{2(s^*-r)}{r(s^*-2)}}+\epsilon^{\frac{(N-1)(1+\beta)(r-2)}{2r}}\\
&\lesssim&\epsilon^{\frac{2(s^*-r)}{r(s^*-2)}}.
\end{eqnarray*}
Thus,% for small $\beta>0$,
\begin{eqnarray*}
\|\int_{J_1}e^{i(\cdot-s)\Delta}\mathcal N\,ds\|_{S^{s_c}(J_1)}
&\lesssim&\sup_{(q,r)\in\Gamma_{s_c}}\|u\|_{L^q(J_1, L^r)}^{2p-1-\theta}\\
&\lesssim&\sup_{(q,r)\in\Gamma_{s_c}}(\|u\|_{L^\infty(J_1, L^r)}\epsilon^{-\frac\beta q})^{2p-1-\theta}\\
&\lesssim&(\epsilon^{\frac{2(s^*-r)}{r(s^*-2)}}\epsilon^{-\frac\beta q})^{2p-1-\theta}\\
&\lesssim&(\epsilon^{\frac{s^*-r}{r(s^*-2)}})^{2p-1-\theta}.
\end{eqnarray*}
Take $(q,r)$ be an $s_c$-admissible pair and the admissible couple $(c,d)\in\Gamma$, 
$$\frac1c:=\frac1{2+N-4s_c}(\frac{2+N}q-\frac{s_c(N-2)}{2}),\quad \frac1d:=\frac{2+N}{2+N-4s_c}\frac1r.$$
Since $\frac{2N}{N-2s_c}<r<\frac{2N}{N-2}$, one has $2<d=\frac{2+N-4s_c}{2+N}r<\frac{2N}{N-2}$. With Strichartz estimates, write
\begin{eqnarray*}
\|\int_{J_2}e^{i(\cdot-s)\Delta}\mathcal N\,ds\|_{L^q((T,\infty),L^r)}
&\lesssim&\|\int_{J_2}e^{i(\cdot-s)\Delta}\mathcal N\,ds\|_{L^c((T,\infty),L^d(\R^N))}^{1-\frac{4s_c}{2+N}}\\
&X&\|\int_{J_2}e^{i(\cdot-s)\Delta}\mathcal N\,ds\|_{L^\frac{8}{N-2}((T,\infty),L^\infty(\R^N))}^{\frac{4s_c}{2+N}}.
\end{eqnarray*}
Using the dispersion of the free NLS kernel via Hardy-Littlewood-Sobolev inequality, one gets
\begin{eqnarray*}
(II)
&:=&\|\int_{J_2}e^{i(\cdot-s)\Delta}\mathcal N\,ds\|_{L^\frac{8}{N-2}((T,\infty),L^\infty(\R^N))}^{\frac{4s_c}{2+N}}\\
&\lesssim&\|\int_{J_2}\frac1{(\cdot-s)^{\frac N2}}\|\mathcal N\|_{L^1_x}\,ds\|_{L^\frac{8}{N-2}(T,\infty)}^{\frac{4s_c}{2+N}}\\
&\lesssim&\|u\|_{L^\infty(H^1)}^{2p-1}\|\int_{J_2}\frac{ds}{(\cdot-s)^{\frac N2}}\,\|_{L^\frac{8}{N-2}(T,\infty)}^{\frac{4s_c}{2+N}}\\
&\lesssim&\|u\|_{L^\infty(H^1)}^{2p-1}\|(\cdot-T+\epsilon^{-\beta})^{-\frac{N-2}2}\|_{L^\frac{8}{N-2}(T,\infty)}^{\frac{4s_c}{2+N}}\\
&\lesssim&\epsilon^\frac{3\beta(N-2)s_c}{2(2+N)}.
\end{eqnarray*}
Moreover,
\begin{eqnarray*}
(I)
&:=&\|\int_{J_1}e^{i(\cdot-s)\Delta}\mathcal N\,ds\|_{L^c((T,\infty),L^d(\R^N))}^{1-\frac{4s_c}{2+N}}\\
%&:=&\|\int_{J_2}e^{i(\cdot-s)\Delta}\mathcal N\,ds\|_{L^\frac{8}{N-2}((T,\infty),L^\infty(\R^N))}^{\frac{4s_c}{2+N}}\\
&=&\|e^{i\cdot\Delta}(e^{i(-T+\epsilon^{-\beta})\Delta}u(T-\epsilon^{-\beta})-u_0)\|_{L^c((T,\infty),L^d(\R^N))}^{1-\frac{4s_c}{2+N}}\\
&\lesssim&\|u\|_{L^\infty((T,\infty),L^2(\R^N))}^{1-\frac{4s_c}{2+N}}.
\end{eqnarray*}
Thus,
$$\|\int_{J_2}e^{i(\cdot-s)\Delta}\mathcal N\,ds\|_{L^q((T,\infty),L^r)}\lesssim\epsilon^\frac{3\beta(N-2)s_c}{2(2+N)}.$$
Taking account of Duhamel formula
$$e^{i(\cdot-T)\Delta}u(T)=e^{i\cdot\Delta}u_0+i\int_{J_1}e^{i(t-s)\Delta}\mathcal N\,ds++i\int_{J_2}e^{i(\cdot-s)\Delta}\mathcal N\,ds.$$
Thus, there exists $\gamma>0$ such that
$$\|e^{i.\Delta}u(T)\|_{S^{s_c}(0,\infty)}=\|e^{i(.-T)\Delta}u(T)\|_{S^{s_c}(T,\infty)}\lesssim\epsilon^\gamma.$$
So, with Lemma \ref{tch} via the absorption result Lemma \ref{abs}, one gets
$$\|u\|_{S^{s_c}(T,\infty)}\lesssim\epsilon^\gamma.$$
With Lemma \ref{tch2}, one gets for $u_+:=e^{-iT\Delta}u(T)+i\int_T^\infty e^{-is\Delta}\mathcal N\,ds$,
\begin{eqnarray*}
\|u(t)-e^{-it\Delta}u_+\|_{H^1}
&=&\|\int_t^\infty e^{i(t-s)\Delta}\mathcal N\,ds\|_{H^1}\\
&\lesssim&\|(1+\nabla)\mathcal N\|_{S'(t,\infty)}\\
&\lesssim&\|u\|_{L^\infty((t,\infty),H^1)}^{\theta_1}\| u\|_{S(t,\infty)}^{2p-\theta_1}+\|u\|_{L^\infty((t,\infty),H^1)}^{\theta_2}\|u\|_{S(t,\infty)}^{2p-\theta_2}.
\end{eqnarray*}
Take two admissible pairs
$$(q,r)\in\Gamma\quad\mbox{and}\quad (q,a)\in\Gamma^{s_c}.$$
Thus, $2<r<a$ and an interpolation gives
$$\|\cdot\|_{L^q(L^r)}\lesssim\|\cdot\|_{L^\infty(L^2)}^{\frac{2(a-r)}{r(a-2)}}\|\cdot\|_{L^q(L^a)}^{1-\frac{2(a-r)}{r(a-2)}}.$$
Thus, as $t\to\infty$, 
$$\|u(t)-e^{-it\Delta}u_+\|_{H^1}\lesssim\| u\|_{S^{s_c}(t,\infty)}^{2p-\theta_1}+\| u\|_{S^{s_c}(t,\infty)}^{2p-\theta_2}\to0.$$
This finishes the proof.
%%%%%%%%%%%%%%%%%%%%%%%%%%%%%%%%%%%%%%%%%%%%%%%%%%%%%%%%%%%%%%%%%%%%%%%%%%%%%%%%%%%%%%%%%%%%%%%%%%%%%%
\section{Morawetz estimate}
%%%%%%%%%%%%%%%%%%%%%%%%%%%%%%%%%%%%%%%%%%%%%%%%%%%%%%%%%%%%%%%%%%%%%%%%%%%%%%%%%%%%%%%%%%%%%%%%%%%%%%%%%
This section is devoted to prove Proposition \ref{cr} about some classical Morawetz estimates satisfied by energy global solutions to the inhomogeneous Choquard problem \eqref{S}. 
%\begin{prop}\label{mor}
%Take $N\geq3$, $b,\alpha$ satisfying \eqref{cnd}, $p_*<p<p^*$ such that $p\geq2$ and $u\in C(\R,H^1)$ be a global solution to \eqref{S}. For $T>0$ and $R:=R(\delta,M(u),\phi)$ large enough. Then, 
%$$\frac1T\int_0^T\Big(\int_{|x|<R}|u(t,x)|^{\frac{2Np}{\alpha+N+2b}}\Big)^{\frac{\alpha+N+2b}{Np}}\,dt\lesssim \frac RT+R^{-\frac{B(N-1)}{2N}}+R^{-2}.$$
%\end{prop}
\begin{proof}[Proof of Proposition \ref{cr}]
Using the properties of $a$ via Cauchy-Schwarz inequality, one has
$$\sup_{t\in\R}M_a(t)\lesssim R.$$
Taking account of proposition \ref{vrnc}, write
\begin{eqnarray*}
M_a'
&=&4\int_{\R^N}\partial_l\partial_ka\Re(\partial_ku\partial_l\bar u)\,dx-\int_{\R^N}\Delta^2a|u|^2\,dx\\
&+&2(\frac2p-1)\int_{\R^N}\Delta a(I_\alpha*|\cdot|^b|u|^{p})|x|^b|u|^p\,dx\\
&+&\frac{4b}p\int_{\R^N}x\nabla a|x|^{b-2}(I_\alpha*|\cdot|^b|u|^p)|u|^{p}\,dx\\
&+&\frac{4}p(\alpha-N)\int_{\R^N}\nabla a(\frac.{|\cdot|^2}I_\alpha*|\cdot|^b|u|^p)|x|^b|u|^{p}\,dx.
\end{eqnarray*}
In the centered ball of radius $\frac R2$, 
\begin{eqnarray*}
(I)
&:=&4\int_{|x|<\frac R2}\partial_l\partial_ka\Re(\partial_ku\partial_l\bar u)\,dx-\int_{|x|<\frac R2}\Delta^2a|u|^2\,dx\\
&+&2(\frac2p-1)\int_{|x|<\frac R2}\Delta a(I_\alpha*|\cdot|^b|u|^{p})|x|^b|u|^p\,dx\\
&+&\frac{4b}p\int_{|x|<\frac R2}x\nabla a|x|^{b-2}(I_\alpha*|\cdot|^b|u|^p)|u|^{p}\,dx\\
%&???&\frac{4}p(\alpha-N)\int_{|x|<\frac R2}\nabla a(\frac.{|\cdot|^2}I_\alpha*|\cdot|^b|u|^p)|x|^b|u|^{p}\,dx\\
&=&8\int_{|x|<\frac R2}|\nabla u|^2\,dx+4N(\frac2p-1)\int_{|x|<\frac R2}(I_\alpha*|\cdot|^b|u|^{p})|x|^b|u|^p\,dx\\
&+&\frac{8b}p\int_{|x|<\frac R2}|x|^{b}(I_\alpha*|\cdot|^b|u|^p)|u|^{p}\,dx.
%&???&\frac{8}p(\alpha-N)\int_{|x|<\frac R2}x(\frac.{|\cdot|^2}I_\alpha*|\cdot|^b|u|^p)|x|^b|u|^{p}\,dx.
\end{eqnarray*}
%????????????????\\Moreover,
%\begin{eqnarray*}
%(B)
%&:=&\int_{|x|<\frac R2}x(\frac.{|\cdot|^2}I_\alpha*|\cdot|^b|u|^p)|x|^b|u|^{p}\,dx\\
%&=&\int_{|x|<\frac R2}\int_{\R^N}(x-y+y)(x-y)\frac{I_\alpha(x-y)}{|x-y|^2}|x|^b|y|^b|u(y)|^p|u(x)|^{p}\,dy\,dx\\
%&=&\int_{|x|<\frac R2}\int_{\R^N}{I_\alpha(x-y)}|x|^b|y|^b|u(y)|^p|u(x)|^{p}\,dy\,dx+\int_{|x|<\frac R2}\int_{\R^N}y(x-y)\frac{I_\alpha(x-y)}{|x-y|^2}|x|^b|y|^b|u(y)|^p|u(x)|^{p}\,dy\,dx\\
%&=&\int_{|x|<\frac R2}(I_\alpha*|\cdot|^b|u|^p)|x|^b|u|^{p}\,dx-\int_{|y|<\frac R2}x(x-y)\frac{I_\alpha(x-y)}{|x-y|^2}|x|^b|y|^b|u(y)|^p|u(x)|^{p}\,dy\,dx\\
%&???&\int_{\R^N}{I_\alpha(x-y)}|x|^b|y|^b|u(y)|^p|u(x)|^{p}\,dy\,dx-(B)\\
%&=&\frac12\int_{\R^N}(I_\alpha*|\cdot|^b|u|^p)|x|^b|u|^{p}\,dx.
%\end{eqnarray*}
%Thus,
%\begin{eqnarray*}
%(I)
%&=&8\int_{|x|<\frac R2}|\nabla u|^2\,dx-\frac{4B}p\int_{|x|<\frac R2}(I_\alpha*|\cdot|^b|u|^{p})|x|^b|u|^p\,dx.
%&+&\frac{4}p(2b+\alpha-N)\int_{|x|<\frac R2}|x|^{b}(I_\alpha*|\cdot|^b|u|^p)|u|^{p}\,dx.
%\end{eqnarray*}
%???????????????????\\
In the centered annulus $C(0,\frac R2,R)$, one has
\begin{eqnarray*}
(II)
&:=&4\int_{\frac R2<|x|<R}\partial_l\partial_ka\Re(\partial_ku\partial_l\bar u)\,dx-\int_{\frac R2<|x|<R}\Delta^2a|u|^2\,dx\\
&+&2(\frac2p-1)\int_{\frac R2<|x|<R}\Delta a(I_\alpha*|\cdot|^b|u|^{p})|x|^b|u|^p\,dx\\
&+&\frac{4b}p\int_{\frac R2<|x|<R}x\nabla a|x|^{b-2}(I_\alpha*|\cdot|^b|u|^p)|u|^{p}\,dx\\
%&+&\frac{4}p(\alpha-N)\int_{\frac R2<|x|<R}\nabla a(\frac.{|\cdot|^2}I_\alpha*|\cdot|^b|u|^p)|x|^b|u|^{p}\,dx\\
&\geq&4\int_{\frac R2<|x|<R}\partial_l\partial_ka\Re(\partial_ku\partial_l\bar u)\,dx-\mathcal O\Big(R\int_{\frac R2<|x|<R}\frac{|u|^2}{|x|^3}\,dx\Big)\\
&+&2(\frac2p-1)R\int_{\frac R2<|x|<R}(I_\alpha*|\cdot|^b|u|^{p})|x|^{b-1}|u|^p\,dx\\
&+&\frac{8b}p\int_{\frac R2<|x|<R}(I_\alpha*|\cdot|^b|u|^p)|x|^{b}|u|^{p}\,dx.
%&+&\frac{8}p(\alpha-N)\int_{|x|<\frac R2}x(\frac.{|\cdot|^2}I_\alpha*|\cdot|^b|u|^p)|x|^b|u|^{p}\,dx.
\end{eqnarray*}
For $|x|>R$, write using the radial assumption on $u$,
\begin{eqnarray*}
(III)
&:=&4\int_{|x|>R}\partial_l\partial_ka\Re(\partial_ku\partial_l\bar u)\,dx-\int_{|x|>R}\Delta^2a|u|^2\,dx\\
&+&2(\frac2p-1)\int_{|x|>R}\Delta a(I_\alpha*|\cdot|^b|u|^{p})|x|^b|u|^p\,dx\\
&+&\frac{4b}p\int_{|x|>R}x\nabla a|x|^{b-2}(I_\alpha*|\cdot|^b|u|^p)|u|^{p}\,dx\\
%&+&\frac{4}p(\alpha-N)\int_{\frac R2<|x|<R}\nabla a(\frac.{|\cdot|^2}I_\alpha*|\cdot|^b|u|^p)|x|^b|u|^{p}\,dx\\
&\geq&-\mathcal O\Big(R\int_{|x|>R}\frac{|u|^2}{|x|^3}\,dx\Big)\\
&+&2(N-1)(\frac2p-1)R\int_{|x|>R}(I_\alpha*|\cdot|^b|u|^{p})|x|^{b-1}|u|^p\,dx\\
&+&\frac{8Rb}p\int_{|x|>R}(I_\alpha*|\cdot|^b|u|^p)|x|^{b-1}|u|^{p}\,dx.
%&+&\frac{8}p(\alpha-N)\int_{|x|<\frac R2}x(\frac.{|\cdot|^2}I_\alpha*|\cdot|^b|u|^p)|x|^b|u|^{p}\,dx.
\end{eqnarray*}
Now, let us define the sets
\begin{gather*}
\Omega:=\{(x,y)\in\R^N\times\R^N,\,\,\mbox{s\,.t}\,\,\frac R2<|x|<R\}\cup \{(x,y)\in\R^N\times\R^N,\,\,\mbox{s\,.t}\,\,\frac R2<|y|<R\};\\
\Omega':=\{(x,y)\in\R^N\times\R^N,\,\,\mbox{s\,.t}\,\,|x|>R,|y|<\frac R2\}\cup \{(x,y)\in\R^N\times\R^N,\,\,\mbox{s\,.t}\,\,|x|<\frac R2, |y|> R\}.
\end{gather*}
%Take also a radial smooth function 
%$$\psi\in C_0^\infty(\R^N),\quad supp(\psi)\subset \{|x|<1\},\quad \psi=1\quad\mbox{on}\quad \{|x|<\frac12\},\quad \psi_R:=\psi(\frac\cdot R).$$
Consider the term
\begin{eqnarray*}
(IV)
%&:=&\frac{4}p(\alpha-N)\int_{|x|>\frac R2}\nabla a(\frac.{|\cdot|^2}I_\alpha*|\cdot|^b|u|^p)|x|^b|u|^{p}\,dx\\
&:=&\int_{\R^N}\nabla a(\frac.{|\cdot|^2}I_\alpha*|\cdot|^b|u|^p)|x|^b|u|^{p}\,dx\\
&=&\frac12\int_{\R^N}\int_{\R^N}(\nabla a(x)-\nabla a(y))(x-y)\frac{I_\alpha(x-y)}{|x-y|^2}|y|^b|u(y)|^p|x|^b|u|^{p}\,dx\,dy\\
&=&\Big(\int_{\Omega}+\int_{\Omega'}+\int_{|x|,|y|<\frac R2}+\int_{|x|,|y|>R}\Big)\Big(\nabla a(x)(x-y)\frac{I_\alpha(x-y)}{|x-y|^2}|y|^b|u(y)|^p|x|^b|u|^{p}\,dx\,dy\Big).
\end{eqnarray*}
Compute
\begin{eqnarray*}
(a)
&:=&\int_{\Omega'}\Big(\nabla a(x)(x-y)\frac{I_\alpha(x-y)}{|x-y|^2}|y|^b|u(y)|^p|x|^b|u|^{p}\Big)\,dx\,dy\\
&=&\int_{\{|x|>R,|y|<\frac R2\}}\Big(\nabla a(x)(x-y)\frac{I_\alpha(x-y)}{|x-y|^2}|y|^b|u(y)|^p|x|^b|u|^{p}\Big)\,dx\,dy\\
&+&\int_{\{|y|>R,|x|<\frac R2\}}\Big(\nabla a(x)(x-y)\frac{I_\alpha(x-y)}{|x-y|^2}|y|^b|u(y)|^p|x|^b|u|^{p}\Big)\,dx\,dy\\
&=&\int_{\{|x|>R,|y|<\frac R2\}}\Big((\nabla a(x)-\nabla a(y))(x-y)\frac{I_\alpha(x-y)}{|x-y|^2}|y|^b|u(y)|^p|x|^b|u|^{p}\Big)\,dx\,dy\\
&=&2\int_{\{|x|>R,|y|<\frac R2\}}\Big([R\frac{x}{2|x|}-y](x-y)\frac{I_\alpha(x-y)}{|x-y|^2}|y|^b|u(y)|^p|x|^b|u|^{p}\Big)\,dx\,dy.
\end{eqnarray*}
Moreover,
\begin{eqnarray*}
(b)
&:=&\frac12\int_{\{|x|<\frac R2,|y|<\frac R2\}}\Big((\nabla a(x)-\nabla a(y))(x-y)\frac{I_\alpha(x-y)}{|x-y|^2}|y|^b|u(y)|^p|x|^b|u|^{p}\Big)\,dx\,dy\\
&=&\frac12\int_{\{|x|<\frac R2,|y|<\frac R2\}}\Big(2(x-y)(x-y)\frac{I_\alpha(x-y)}{|x-y|^2}|y|^b|u(y)|^p|x|^b|u|^{p}\Big)\,dx\,dy\\
&=&\int_{\{|x|<\frac R2,|y|<\frac R2\}}\Big(I_\alpha(x-y)|y|^b|u(y)|^p|x|^b|u|^{p}\Big)\,dx\,dy\\
&=&\int_{\R^N}(I_\alpha*|\cdot|^b|\psi_Ru|^p)|x|^b|\psi_Ru|^{p}\,dx.
%&=&\int_{\{|x|>R,|y|<\frac R2\}}\Big((\nabla a(x)-\nabla a(y))(x-y)\frac{I_\alpha(x-y)}{|x-y|^2}|y|^b|u(y)|^p|x|^b|u|^{p}\Big)\,dx\,dy\\
%&=&2\int_{\{|x|>R,|y|<\frac R2\}}\Big([R\frac{x}{2|x|}-y](x-y)\frac{I_\alpha(x-y)}{|x-y|^2}|y|^b|u(y)|^p|x|^b|u|^{p}\Big)\,dx\,dy.
\end{eqnarray*}
%%%%%%%%%%%%%%%%%%%%%%%%%%%%%%%%%%%%%%%%%%%%%%%%%%%%%%%%%%%%%
Furthermore,
\begin{eqnarray*}
(c)
&:=&\int_{\{\frac R2<|x|<R\}}\int_{\R^N}\Big(\nabla a(x)(x-y)\frac{I_\alpha(x-y)}{|x-y|^2}|y|^b|u(y)|^p|x|^b|u|^{p}\Big)\,dx\,dy\\
&=&\int_{\{\frac R2<|x|<R,|y-x|>\frac R4\}}\Big(\nabla a(x)(x-y)\frac{I_\alpha(x-y)}{|x-y|^2}|y|^b|u(y)|^p|x|^b|u|^{p}\Big)\,dx\,dy\\
&+&\int_{\{\frac R2<|x|<R,|y-x|<\frac R4\}}\Big(\nabla a(x)(x-y)\frac{I_\alpha(x-y)}{|x-y|^2}|y|^b|u(y)|^p|x|^b|u|^{p}\Big)\,dx\,dy\\
&=&\mathcal O\Big(\int_{\{|x|>\frac R2\}}I_\alpha*|\cdot|^b|u|^p)|x|^b|u|^{p}\,dx\Big).
%&=&\int_{\R^N}(I_\alpha*|\cdot|^b|\psi_Ru|^p)|x|^b|\psi_Ru|^{p}\,dx.
%&=&\int_{\{|x|>R,|y|<\frac R2\}}\Big((\nabla a(x)-\nabla a(y))(x-y)\frac{I_\alpha(x-y)}{|x-y|^2}|y|^b|u(y)|^p|x|^b|u|^{p}\Big)\,dx\,dy\\
%&=&2\int_{\{|x|>R,|y|<\frac R2\}}\Big([R\frac{x}{2|x|}-y](x-y)\frac{I_\alpha(x-y)}{|x-y|^2}|y|^b|u(y)|^p|x|^b|u|^{p}\Big)\,dx\,dy.
\end{eqnarray*}
The last term is
\begin{eqnarray*}
(d)
&:=&\int_{\{|x|,|y|>R\}}\Big(\nabla a(x)(x-y)\frac{I_\alpha(x-y)}{|x-y|^2}|y|^b|u(y)|^p|x|^b|u|^{p}\Big)\,dx\,dy\\
&=&\frac R2\int_{\{|x|,|y|>R\}}\Big(\frac x{|x|}-\frac y{|y|}\Big)(x-y)\frac{I_\alpha(x-y)}{|x-y|^2}|y|^b|u(y)|^p|x|^b|u|^{p}\,dx\,dy\\
&=&\frac R2\int_{\{|x|,|y|>R\}}\Big(\frac{x-y}{|x|}+y(\frac{|y|-|x|}{|x||y|})\Big)(x-y)\frac{I_\alpha(x-y)}{|x-y|^2}|y|^b|u(y)|^p|x|^b|u|^{p}\,dx\,dy\\
%&=&\int_{\{\frac R2<|x|<R,|y-x|>\frac R4\}}\Big(\nabla a(x)(x-y)\frac{I_\alpha(x-y)}{|x-y|^2}|y|^b|u(y)|^p|x|^b|u|^{p}\Big)\,dx\,dy\\
%&+&\int_{\{\frac R2<|x|<R,|y-x|<\frac R4\}}\Big(\nabla a(x)(x-y)\frac{I_\alpha(x-y)}{|x-y|^2}|y|^b|u(y)|^p|x|^b|u|^{p}\Big)\,dx\,dy\\
&\lesssim&\int_{\{|x|> R\}}(I_\alpha*|\cdot|^b|u|^p)|x|^b|u|^{p}\,dx.
\end{eqnarray*}
Using the identity 
$$\int_{\R^N}\psi_R^2|\nabla u|^2\,dx=\int_{\R^N}\Big(|\nabla(\psi_R u)|^2+\psi_R\Delta\psi_R|u|^2\Big)\,dx,$$
%one writes\\?????????
%\begin{eqnarray*}
%(I)
%&=&8\int_{|x|<\frac R2}|\nabla u|^2\,dx-\frac{4B}p\int_{|x|<\frac R2}(I_\alpha*|\cdot|^b|u|^{p})|x|^b|u|^p\,dx\\
%&=&8\int_{|x|<\frac R2}|\nabla(\psi_R u)|^2\,dx+8\int_{|x|<\frac R2}\psi_R\Delta\psi_R|u|^2\,dx-\frac{4B}p\int_{|x|<\frac R2}(I_\alpha*|\cdot|^b|u|^{p})|x|^b|u|^p\,dx.
%\end{eqnarray*}
 and regrouping the previous estimates, yields
{\small\begin{eqnarray*}
M_a'
&\geq&8\int_{|x|<\frac R2}|\nabla(\psi_R u)|^2\,dx+8\int_{|x|<\frac R2}\psi_R\Delta\psi_R|u|^2\,dx-\frac{4N}p(p-2-\frac{2b}N)\int_{|x|<\frac R2}(I_\alpha*|\cdot|^b|u|^{p})|x|^b|u|^p\,dx\\
&+&4\int_{\frac R2<|x|<R}\partial_l\partial_ka\Re(\partial_ku\partial_l\bar u)\,dx-\mathcal O\Big(R\int_{\frac R2<|x|<R}\frac{|u|^2}{|x|^3}\,dx\Big)\\
&+&2(\frac2p-1)R\int_{\frac R2<|x|<R}(I_\alpha*|\cdot|^b|u|^{p})|x|^{b-1}|u|^p\,dx\\
&+&\frac{8b}p\int_{\frac R2<|x|<R}(I_\alpha*|\cdot|^b|u|^p)|x|^{b}|u|^{p}\,dx\\
&-&\mathcal O\Big(R\int_{|x|>R}\frac{|u|^2}{|x|^3}\,dx\Big)\\
&+&2(N-1)(\frac2p-1)R\int_{|x|>R}(I_\alpha*|\cdot|^b|u|^{p})|x|^{b-1}|u|^p\,dx\\
&+&\frac{8Rb}p\int_{|x|>R}(I_\alpha*|\cdot|^b|u|^p)|x|^{b-1}|u|^{p}\,dx\\
&+&\frac4p(\alpha-N)(IV).
\end{eqnarray*}}
So,
\begin{eqnarray*}
M_a'
&\geq&8\int_{|x|<\frac R2}|\nabla(\psi_R u)|^2\,dx-\frac{4N}p(p-2-\frac{2b}N)\int_{|x|<\frac R2}(I_\alpha*|\cdot|^b|u|^{p})|x|^b|u|^p\,dx\\
%&+&4\int_{\frac R2<|x|<R}\partial_l\partial_ka\Re(\partial_ku\partial_l\bar u)\,dx-\frac c{R^2}M(u)\\
%&+&2(\frac2p-1)R\int_{\frac R2<|x|<R}(I_\alpha*|\cdot|^b|u|^{p})|x|^{b-1}|u|^p\,dx\\
&+&\frac{8b}p\int_{\frac R2<|x|<R}(I_\alpha*|\cdot|^b|u|^p)|x|^{b}|u|^{p}\,dx-c\int_{|x|>\frac R2}(I_\alpha*|\cdot|^b|u|^p)|x|^{b-1}|u|^{p}\,dx\\
&+&\frac4p(\alpha-N)(IV)-\frac c{R^2}M(u)\\
&\geq&8\int_{|x|<\frac R2}|\nabla(\psi_R u)|^2\,dx-\frac{4N}p(p-2-\frac{2b}N)\int_{|x|<\frac R2}(I_\alpha*|\cdot|^b|u|^{p})|x|^b|u|^p\,dx\\
&+&\frac{8b}p\int_{\frac R2<|x|<R}(I_\alpha*|\cdot|^b|u|^p)|x|^{b}|u|^{p}\,dx-c\int_{|x|>\frac R2}(I_\alpha*|\cdot|^b|u|^p)|x|^{b-1}|u|^{p}\,dx\\
&-&\frac4p(N-\alpha)\Big(\int_{\R^N}(I_\alpha*|\cdot|^b|\psi_Ru|^p)|x|^b|\psi_Ru|^{p}\,dx+(a)+(c)+(d) \Big)-\frac c{R^2}M(u).
\end{eqnarray*}
Then,
\begin{eqnarray*}
M_a'
&\geq&8\int_{|x|<\frac R2}|\nabla(\psi_R u)|^2\,dx-\frac{4B}p\int_{|x|<\frac R2}(I_\alpha*|\cdot|^b|\psi_Ru|^{p})|x|^b|\psi_Ru|^p\,dx\\
&+&\frac{8b}p\int_{\frac R2<|x|<R}(I_\alpha*|\cdot|^b|u|^p)|x|^{b}|u|^{p}\,dx-c\int_{|x|>\frac R2}(I_\alpha*|\cdot|^b|u|^p)|x|^{b-1}|u|^{p}\,dx\\
&-&\frac4p(N-\alpha)\Big((a)+(c) \Big)-\frac c{R^2}M(u)\\
%&\geq&8\int_{|x|<\frac R2}|\nabla(\psi_R u)|^2\,dx-\frac{4B}p\int_{|x|<\frac R2}(I_\alpha*|\cdot|^b|\psi_Ru|^{p})|x|^b|\psi_Ru|^p\,dx\\
%&+&\frac{8b}p\int_{\frac R2<|x|<R}(I_\alpha*|\cdot|^b|u|^p)|x|^{b}|u|^{p}\,dx-c\int_{|x|>\frac R2}(I_\alpha*|\cdot|^b|u|^p)|x|^{b}|u|^{p}\,dx\\
%&-&\frac4p(N-\alpha)(a)-\frac c{R^2}M(u)\\
&\geq&8\int_{|x|<\frac R2}|\nabla(\psi_R u)|^2\,dx-\frac{4B}p\int_{|x|<\frac R2}(I_\alpha*|\cdot|^b|\psi_Ru|^{p})|x|^b|\psi_Ru|^p\,dx\\
&-&c\int_{\{|x|>R,|y|<\frac R2\}}\Big([R\frac{x}{2|x|}-y](x-y)\frac{I_\alpha(x-y)}{|x-y|^2}|y|^b|u(y)|^p|x|^b|u|^{p}\Big)\,dx\,dy\\
&-&\frac c{R^2}M(u)-c\int_{|x|>\frac R2}(I_\alpha*|\cdot|^b|u|^p)|x|^{b}|u|^{p}\,dx.
\end{eqnarray*}
By Lemma \ref{crcv}, one gets
\begin{eqnarray*}
M_a'
&\geq&8\delta'\|\psi_Ru\|_{\frac{2Np}{N+\alpha+2b}}^2-\frac c{R^2}M(u)-c\int_{|x|>\frac R2}(I_\alpha*|\cdot|^b|u|^p)|x|^{b}|u|^{p}\,dx\\
&-&c\int_{\{|x|>R,|y|<\frac R2\}}\Big([R\frac{x}{2|x|}-y](x-y)\frac{I_\alpha(x-y)}{|x-y|^2}|y|^b|u(y)|^p|x|^b|u|^{p}\Big)\,dx\,dy.
\end{eqnarray*}
Moreover, since for large $R>0$ on $\{|x|>R,|y|<\frac R2\}$, $|x-y|\simeq |x|>R>>\frac R2>|y|$, one has
\begin{eqnarray*}
(d)
&=&\int_{\{|x|>R,|y|<\frac R2\}}\Big([R\frac{x}{2|x|}-y](x-y)\frac{I_\alpha(x-y)}{|x-y|^2}|y|^b|u(y)|^p|x|^b|u|^{p}\Big)\,dx\,dy\\
&\lesssim&\int_{\{|x|>R,|y|<\frac R2\}}\Big([\frac{|x|}{2}+|y|]|x-y|\frac{I_\alpha(x-y)}{|x-y|^2}|y|^b|u(y)|^p|x|^b|u|^{p}\Big)\,dx\,dy\\
&\lesssim&\int_{\{|x|>R,|y|<\frac R2\}}\Big({I_\alpha(x-y)}|y|^b|u(y)|^p|x|^b|u|^{p}\Big)\,dx\,dy\\
&\lesssim&\int_{\R^N}\int_{\R^N}\Big({I_\alpha(x-y)}\chi_{|x|>R}|y|^b|u(y)|^p|x|^b|u|^{p}\Big)\,dx\,dy\\
&\lesssim&\int_{\R^N}\int_{\R^N}\Big({I_\alpha(x-y)}\chi_{|x|>R}|y|^b|u(y)|^p|x|^b|u|^{p}\Big)\,dx\,dy.
\end{eqnarray*}
Taking account of Hardy-Littlewood-Sobolev inequality, H\"older and Strauss estimates and Sobolev injections via the fact that $p_*<p<p^*$, write for $\mu:=(\frac N{|b|})^+$,
\begin{eqnarray*}
(d)
&\lesssim&\int_{\R^N}\int_{\R^N}\Big({I_\alpha(x-y)}\chi_{|x|>R}|y|^b|u(y)|^p|x|^b|u|^{p}\Big)\,dx\,dy\\
&\lesssim&\|u\|_{L^{\frac{2Np}{N+\alpha+2b}}(|x|>R)}^p\|u\|_{L^{\frac{2Np}{N+\alpha+2b}}}^p\\
&\lesssim&\Big(\int_{\{|x|>R\}}|u|^{\frac{2Np}{N+\alpha+2b}}\,dx\Big)^{\frac{N+\alpha+2b}{2N}}.
\end{eqnarray*}
So,
\begin{eqnarray*}
(d)
&\lesssim&\Big(\int_{\{|x|>R\}}|u|^2(|x|^{-\frac{N-1}2}\|u\|^\frac12\|\nabla u\|^\frac12)^{-2+\frac{2Np}{N+\alpha+2b}}\,dx\Big)^{\frac{N+\alpha+2b}{2N}}\\
&\lesssim&\|u\|^{\frac{N+\alpha+2b}{N}}\frac1{R^{\frac{B(N-1)}{2N}}}(\|u\|\|\nabla u\|)^{\frac B{2N}}\\
&\lesssim&R^{-\frac{B(N-1)}{2N}}.
\end{eqnarray*}
Thus, 
$$8\delta'\|\psi_Ru\|_{\frac{2Np}{N+\alpha+2b}}^2
\leq M_a'+\frac c{R^2}M(u)+c\int_{|x|>\frac R2}(I_\alpha*|\cdot|^b|u|^p)|x|^{b}|u|^{p}\,dx+R^{-\frac{B(N-1)}{2N}}.
$$
By the fundamental theorem of calculus and arguing as previously, one gets
\begin{eqnarray*}
\frac1T\int_0^T\|\psi_Ru\|_{\frac{2Np}{N+\alpha+2b}}^2\,ds
&\lesssim&\frac{\|M_a\|_{L^\infty([0,T])}}T+\frac cT\int_0^T\int_{|x|>\frac R2}(I_\alpha*|\cdot|^b|u|^p)|x|^{b}|u|^{p}\,dx\,ds\\
&+&R^{-\frac{B(N-1)}{2N}}+\frac{c}{R^2}\\
&\leq&\frac{\|M_a\|_{L^\infty([0,T])}}T+R^{-\frac{B(N-1)}{2N}}+R^{-2}\\
&\lesssim&\frac RT+R^{-\frac{B(N-1)}{2N}}+R^{-2}.
\end{eqnarray*}
\end{proof}
Let us prove an energy evacuation.
\begin{lem}\label{evac}
Take $N\geq3$, $b,\alpha$ satisfying \eqref{cnd}, $p_*<p<p^*$ such that $p\geq2$ and $u\in C(\R,H^1)$ be a global solution to \eqref{S}. Then, there exists two sequences of real numbers $t_n\to\infty, R_n\to\infty$ such that
$$\|u(t_n)\|_{L^\frac{2Np}{N+\alpha+2b}(|x|\leq R_n)}\to0\quad\mbox{as}\quad n\to\infty.$$
\end{lem}
\begin{proof}
By Proposition \ref{cr}, if $\frac{B(N-1)}{2N}>2$, taking $R:=\sqrt T>>1$, one gets 
\begin{eqnarray*}
\frac1T\int_0^T\|u(s)\|_{L^\frac{2Np}{N+\alpha+2b}(|x|<\sqrt s)}^2\,ds
&\leq&\frac1T\int_0^T\|u(s)\|_{L^\frac{2Np}{N+\alpha+2b}(\sqrt{T})}^2\,ds\\
&\lesssim&\frac1{\sqrt{T}}+\frac1T\\
&\lesssim&\frac1{\sqrt{T}}.
\end{eqnarray*}
The proof follows with the mean value Theorem because $u\in L^\infty(\R, L^\frac{2Np}{N+\alpha+2b})$. The second case $\frac{B(N-1)}{2N}<2$ is similar.
\end{proof}
%%%%%%%%%%%%%%%%%%%%%%%%%%%%%%%%%%%%%%%%%%%%%%%%%%%%%%%%%%%%%%%%%%%%%%%%%%%%%%%%%%%%%%%
\section{Proof of the main result}
%%%%%%%%%%%%%%%%%%%%%%%%%%%%%%%%%%%%%%%%%%%%%%%%%%%%%%%%%%%%%%%%%%%%%%%%%%%%%%%%%%%%%%%%%
This section is devoted to prove Theorem \ref{sctr}. Taking account of Theorem 6.1 in \cite{saa}, it follows that $T^*=\infty$ and 
$$\mathcal G\mathcal M(u(t))<1, \quad \forall t\in\R.$$
Take two sequences of real numbers  $t_n,R_n\to\infty$ given by Proposition \ref{evac}. Applying H\"older inequality for large $n$ so that $R_n>R>0$, one gets as $n\to\infty$,
\begin{eqnarray*}
\|u(t_n)\|_{L^2(|x|<R)}
&\leq&(|\{|x|<R\}|)^{\frac B{2Np}}\|u(t_n)\|_{L^{\frac{2Np}{N+\alpha+2b}}(|x|<R_n)}\\
&\lesssim&R^{\frac B{2p}}\|u(t_n)\|_{L^{\frac{2Np}{N+\alpha+2b}}(|x|<R_n)}\to0.
\end{eqnarray*}
The scattering follows by Proposition \ref{crt}.
%%%%%%%%%%%%%%%%%%%%%%%%%%%%%%%%%%%%%%%%%%%%%%%%%%%%%%%%%%%%%%%%%%%%%%%%%%%%%%%%%%%%%%%
\section{Appendix}
%%%%%%%%%%%%%%%%%%%%%%%%%%%%%%%%%%%%%%%%%%%%%%%%%%%%%%%%%%%%%%%%%%%%%%%%%%%%%%%%%%%%%%%%%
%%%%%%%%%%%%%%%%%%%%%%%%%%%%%%%%%%%%%%%%%%%%%%%%%%%%%%%%%%%%%
This section is devoted to prove a variance identity and a global Strichartz type estimate.
%%%%%%%%%%%%%%%%%%%%%%%%%%%%%%%%%%%%%%%%%%%%%%%%%
\subsection{Proof of Proposition \ref{vrnc}}
%%%%%%%%%%%%%%%%%%%%%%%%%%%%%%%%%%%%%%%%%%%%%%%%%%%%%%%%%%%%%%%%%%%%%%%%%%%%%%%%%%%%%
Let ${u}\in C_{T^*}(H^1)$ be a solution to \eqref{S}. %Define the real function on $[0,T^*)$, by
%$$V:=\int_{\R^N}|xu(.,x)|^2\,dx.$$
Denote the quantities
$$V_a:=\int_{\R^N}a(x)|u(.,x)|^2\,dx\quad\mbox{and}\quad\mathcal N:=(I_\alpha*|\cdot|^b|u|^p)|x|^b|u|^{p-2}u.$$
Multiplying the equation \eqref{S} by $2u$ and examining the imaginary parts, one gets
$$\partial_t (|u|^2) =-2\Im(\bar u \Delta u).$$% = −2\nabla· (\Im(\bar u) \nabla u)= −2div(\Im(\bar u) \nabla u).$$ 
Thus, 
$$V'_a=-2\int_{\R^N}a(x)\Im(\bar u\Delta u)\,dx=2\Im\int_{\R^N}(\partial_ka\partial_ku)\bar u \,dx=M_a.$$
Compute, 
\begin{eqnarray*}
\partial_t\Im(\partial_k u\bar u)
&=&\Im(\partial_k\dot u\bar u)+\Im(\partial_k u\bar{\dot u})\\
&=&\Re(i\dot u\partial_k\bar u)-\Re(i\partial_k \dot u\bar{u})\\
&=&\Re(\partial_k\bar u(-\Delta u-\mathcal N))-\Re(\bar u\partial_k(-\Delta u-\mathcal N))\\
&=&\Re(\bar u\partial_k\Delta u-\partial_k\bar u\Delta u)+\Re(\bar u\partial_k\mathcal N-\partial_k\bar u\mathcal N).
\end{eqnarray*}
Thanks to the identity,
$$\frac12\partial_k\Delta(|u|^2)-2\partial_l\Re(\partial_{k}u\partial_l\bar u)=\Re(\bar u\partial_k\Delta u-\partial_k\bar u\Delta u),$$
 one has
\begin{eqnarray*}
\int_{\R^N}\partial_ka\Re(\bar u\partial_k\Delta u-\partial_k\bar u\Delta u)\,dx
&=&\int_{\R^N}\partial_ka\Big(\frac12\partial_k\Delta(|u|^2)-2\partial_l\Re(\partial_ku\partial_l\bar u)\Big)\,dx\\
&=&2\int_{\R^N}\partial_l\partial_ka\Re(\partial_ku\partial_l\bar u)\,dx-\frac12\int_{\R^N}\Delta^2a|u|^2\,dx.
\end{eqnarray*}
On the other hand
\begin{eqnarray*}
\Re(\bar u\partial_k\mathcal N)
&=&\Re(\bar u\partial_k[(I_\alpha*|\cdot|^b|u|^p)|x|^b|u|^{p-2}u])\\
&=&\Re(\bar u[(\alpha-N)(\frac{x_k}{|\cdot|^2}I_\alpha*|\cdot|^b|u|^p)|x|^b|u|^{p-2}u]+[(I_\alpha*|\cdot|^b|u|^p)|x|^b|u|^{p-2}\partial_ku])\\
&+&\Re(\bar u(I_\alpha*|\cdot|^b|u|^p)[bx_k|x|^{b-2}|u|^{p-2}u]+(p-2)[(I_\alpha*|\cdot|^b|u|^p)|x|^b\Re(\partial_ku\bar u)|u|^{p-4}u])\\
&=&(\alpha-N)(\frac{x_k}{|\cdot|^2}I_\alpha*|\cdot|^b|u|^p)|x|^b|u|^{p}+(I_\alpha*|\cdot|^b|u|^p)|x|^b|u|^{p-2}\Re(\partial_k\bar u u)\\
&+&bx_k(I_\alpha*|\cdot|^b|u|^p)|x|^{b-2}|u|^{p}+(p-2)(I_\alpha*|\cdot|^b|u|^p)|x|^b|u|^{p-2}\Re(\partial_ku\bar u).
%&=&-\int_{\R^N}\Delta a\bar u\mathcal N+2\int_{\R^N}\Re(\partial_ka\partial_k\bar u\mathcal N)\Big)\,dx\\
%&=&-2N\int_{\R^N}(I_\alpha*|\cdot|^b|u|^{p})|x|^b|u|^p\,dx-2\int_{\R^N}\partial_ka\Re(\partial_k\bar u\mathcal N)\,dx\\
%&=&-\int_{\R^N}\Delta a\bar u\mathcal N\,dx-\frac2{p}\int_{\R^N}\partial_ka\partial_k(|u|^{p})(I_\alpha*|\cdot|^b|u|^p)|x|^b\,dx.
\end{eqnarray*}
Thus,
\begin{eqnarray*}
(A)
&:=&\int_{\R^N}\partial_ka\Re(\bar u\partial_k\mathcal N-\partial_k\bar u\mathcal N)\,dx\\
&=&b\int_{\R^N}\partial_kax_k(I_\alpha*|\cdot|^b|u|^p)|x|^{b-2}|u|^{p}\,dx+(p-2)\int_{\R^N}\partial_ka(I_\alpha*|\cdot|^b|u|^p)|x|^b|u|^{p-2}\partial_k(|u|^2)\,dx\\
&+&(\alpha-N)\int_{\R^N}\partial_ka(\frac{x_k}{|\cdot|^2}I_\alpha*|\cdot|^b|u|^p)|x|^b|u|^{p}\,dx\\%+(I_\alpha*|\cdot|^b|u|^p)|x|^b|u|^{p-2}\Re(\partial_k\bar u u)\\
&=&b\int_{\R^N}\partial_kax_k(I_\alpha*|\cdot|^b|u|^p)|x|^{b-2}|u|^{p}\,dx+\frac{p-2}p\int_{\R^N}\partial_ka(I_\alpha*|\cdot|^b|u|^p)|x|^b\partial_k(|u|^p)\,dx\\
&+&(\alpha-N)\int_{\R^N}\partial_ka(\frac{x_k}{|\cdot|^2}I_\alpha*|\cdot|^b|u|^p)|x|^b|u|^{p}\,dx.%+(I_\alpha*|\cdot|^b|u|^p)|x|^b|u|^{p-2}\Re(\partial_k\bar u u)\\
\end{eqnarray*}
Write
\begin{eqnarray*}
(B)
&:=&\int_{\R^N}\partial_ka(I_\alpha*|\cdot|^b|u|^p)|x|^b\partial_k(|u|^p)\,dx\\
&=&-\int_{\R^N}\Delta a(I_\alpha*|\cdot|^b|u|^p)|x|^b|u|^p\,dx-(\alpha-N)\int_{\R^N}\partial_k a(\frac{x_k}{|\cdot|^2}I_\alpha*|\cdot|^b|u|^p)|x|^b|u|^p\,dx\\
&-&b\int_{\R^N}\partial_ka(I_\alpha*|\cdot|^b|u|^p)x_k|x|^{b-2}|u|^p\,dx.
\end{eqnarray*}
So,
\small{\begin{eqnarray*}
(A)
&=&(\alpha-N)\int_{\R^N}\partial_ka(\frac{x_k}{|\cdot|^2}I_\alpha*|\cdot|^b|u|^p)|x|^b|u|^{p}\,dx+b\int_{\R^N}\partial_kax_k(I_\alpha*|\cdot|^b|u|^p)|x|^{b-2}|u|^{p}\,dx\\
&+&\frac{p-2}p(B)\\
&=&(\alpha-N)\int_{\R^N}\partial_ka(\frac{x_k}{|\cdot|^2}I_\alpha*|\cdot|^b|u|^p)|x|^b|u|^{p}\,dx+b\int_{\R^N}\partial_kax_k(I_\alpha*|\cdot|^b|u|^p)|x|^{b-2}|u|^{p}\,dx\\
&-&\frac{p-2}p(\int_{\R^N}\Delta a(I_\alpha*|\cdot|^b|u|^p)|x|^b|u|^p\,dx+(\alpha-N)\int_{\R^N}\partial_k a(\frac{x_k}{|\cdot|^2}I_\alpha*|\cdot|^b|u|^p)|x|^b|u|^p\,dx\\
&+&b\int_{\R^N}\partial_ka(I_\alpha*|\cdot|^b|u|^p)x_k|x|^{b-2}|u|^p)\,dx\\
&=&\frac{2}p\Big((\alpha-N)\int_{\R^N}\partial_ka(\frac{x_k}{|\cdot|^2}I_\alpha*|\cdot|^b|u|^p)|x|^b|u|^{p}\,dx+b\int_{\R^N}\partial_kax_k(I_\alpha*|\cdot|^b|u|^p)|x|^{b-2}|u|^{p}\,dx\Big)\\
&-&\frac{p-2}p\int_{\R^N}\Delta a(I_\alpha*|\cdot|^b|u|^p)|x|^b|u|^p\,dx.
\end{eqnarray*}}
Finally,
\small{\begin{eqnarray*}
V''_a
&=&4\int_{\R^N}\partial_l\partial_ka\Re(\partial_ku\partial_l\bar u)\,dx-\int_{\R^N}\Delta^2a|u|^2\,dx\\
&+&\frac{4}p\Big((\alpha-N)\int_{\R^N}\partial_ka(\frac{x_k}{|\cdot|^2}I_\alpha*|\cdot|^b|u|^p)|x|^b|u|^{p}\,dx+b\int_{\R^N}\partial_kax_k(I_\alpha*|\cdot|^b|u|^p)|x|^{b-2}|u|^{p}\,dx\Big)\\
&-&\frac{2(p-2)}p\int_{\R^N}\Delta a(I_\alpha*|\cdot|^b|u|^p)|x|^b|u|^p\,dx.
\end{eqnarray*}}
This completes the proof.
%%%%%%%%%%%%%%%%%%%%%%%%%%%%%%%%%%%%%%%%%%%%%%%%%%%%%%%%%%%%%
\subsection{Proof of Lemma \ref{tch2}}
%%%%%%%%%%%%%%%%%%%%%%%%%%%%%%%%%%%%%%%%%%%%%%%%%%%%%%%%%%%%%%%%%%%%%%%%%%%%%%%%%%%%%
Denote, for $D\subset\R^N$, $S_T(D):=\cap_{(q,r)\in\Gamma}L^q_T(L^r(D)).$ With Duhamel formula and Strichartz estimates, one gets
\begin{eqnarray*}
(E)
&:=&\|\left\langle u-e^{i.\Delta}u_0\right\rangle\|_{S(0,T)}\\%\lesssim\\
%&\lesssim&\|[I_\alpha*|\cdot|^b|u|^p]|x|^b|u|^{p-1}\|_{S_T'(\R^N)}+\|\nabla\Big([I_\alpha*|\cdot|^b|u|^p]|x|^b|u|^{p-1}\Big)\|_{S_T'(\R^N)}\\
&\lesssim&\|[I_\alpha*|\cdot|^b|u|^p]|x|^b|u|^{p-1}\|_{S_T'(|x|<1)}+\|\nabla\Big([I_\alpha*|\cdot|^b|u|^p]|x|^b|u|^{p-2}u\Big)\|_{S_T'(|x|<1)}\\
&+&\|[I_\alpha*|\cdot|^b|u|^p]|x|^b|u|^{p-1}\|_{S_T'(|x|>1)}+\|\nabla\Big([I_\alpha*|\cdot|^b|u|^p]|x|^b|u|^{p-2}u\Big)\|_{S_T'(|x|>1)}\\
&:=&(A)+(B)+(C)+(D).
\end{eqnarray*}
Take $\mu_1:=(\frac N{-b})^-$, $r_1:=\frac{2Np}{\alpha+N-\frac{2N}{\mu_1}}$ and $(q_1,r_1)\in\Gamma$. Then, $1+\frac\alpha N=\frac2{\mu_1}+\frac{2p}{r_1}$ and using H\"older and Hardy-Littlewood-Sobolev inequalities, one gets for $\theta_1:=q_1-2$,
\begin{eqnarray*}
(A)
&\lesssim&\|(I_\alpha*|\cdot|^b|u|^p)|x|^b|u|^{p-1}\|_{L^{q_1'}_T(L^{r'_1}(|x|<1))}\\\\
&\lesssim&\||x|^b\|_{L^\mu(|x|<1)}^2\|\|u\|_{r_1}^{2p-1}\|_{L^{q'_1}(0,T)}\\
&\lesssim&\|u\|_{L^\infty_T(L^{r_1})}^{2(p-1)-\theta_1}\|u\|_{L^{q_1}_T(L^{r_1})}^{1+\theta_1},
\end{eqnarray*}
Since $p>p_*$, it follows that $(2p-1)q'_1>q_1$ and 
$$0<\theta_1<2(p-1).$$
Moreover, $p_*<p<p^*$ implies that $2<r_1<\frac{2N}{N-2}.$ Now, compute
\begin{eqnarray*}
(F)&:=&|\nabla((I_\alpha *|\cdot|^b|u|^{p})|x|^b|u|^{p-2}u)|\\
&\lesssim& |(I_\alpha*|\cdot|^b|u|^p)|x|^b\nabla u|u|^{p-2}|+|(I_\alpha *|\cdot|^b\Re(\nabla u\bar u)|u|^{p-2})|x|^b|u|^{p-1}|\\
&+&|(I_\alpha*|\cdot|^b|u|^p)|x|^{b-1}|u|^{p-1}|+|(I_\alpha*|\,.\,|^{b-1}|u|^p)|x|^b|u|^{p-1}|.
\end{eqnarray*}
There exist four terms to control. The two first ones can be controlled as previously. 
\begin{eqnarray*}
(B_1)
&:=&\|(I_\alpha*|\cdot|^b|u|^p)|x|^b\nabla u|u|^{p-2}\|_{S_T'(|x|<1)}+\|(I_\alpha *|\cdot|^b\Re(\nabla u\bar u)|u|^{p-2})|x|^b|u|^{p-1}\|_{S_T'(|x|<1)}\\
&\lesssim&\|u\|_{L^\infty_T(L^{r_1})}^{2(p-1)-\theta_1}\|u\|_{L^{q_1}_T(W^{1,r_1})}^{1+\theta_1}.
\end{eqnarray*}
Choosing $\rho:=(\frac N{1-2b})^-$ and $r_2:=\Big(\frac{2Np}{N+\alpha+2b}\Big)^+$, one gets
$$1+\frac\alpha N=\frac{1}{r_2}+\frac1{\rho}+\frac{2(p-1)}{r_2}+\frac{1}{r_2}-\frac1N.$$
%Since $p<p^*$, it follows that $(2p-1)q_1'<q_1$ and there exists a positive number denoted also $\delta>0$ such that $\frac1{q_1'}=\frac{2p-1}{q_1}+\frac1\delta$. 
 Since $\frac1\rho:=(\frac{1-b}N)^++(\frac{-b}N)^+$, taking account of Hardy-Littlewood-Sobolev and H\"older inequalities, one obtains
\begin{eqnarray*}
(B_2)
&:=&\||x|^{b-1}(I_\alpha*|\cdot|^b|u|^p)|u|^{p-1}\|_{L^{q_2'}_T(L^{r_2'}(|x|<1))}\\
&+&\|(I_\alpha*|\cdot|^{b-1}|u|^p)|x|^{b}|u|^{p-1}\|_{L^{q_2'}_T(L^{r_2'}(|x|<1))}\\
&\lesssim&\|\||x|^b\|_{L^a(|x|<1)}\||x|^{b-1}\|_{L^{a'}(|x|<1)}\|u\|_{r_2}^{2(p-1)}\|u\|_{\frac{Nr_2}{N-r_2}}\|_{L^{q_2'}(0,T)}\\
&\lesssim&\|u\|_{L^\infty_T(L^{r_2})}^{2(p-1)-\theta_2}\|u\|_{L^{q_2}_T(W^{1,r_2})}^{1+\theta_2},
\end{eqnarray*}
where $(1+\theta_2)q_2'=q_2$. The condition $p_*<p<p^*$ gives $2<q_2<2p$ which is equivalent to $0<\theta_2<2(p-1)$.\\
Let us estimate $(C)$. 
Take $\mu_+:=(\frac N{-b})^+$, $r_+:=\frac{2Np}{\alpha+N-\frac{2N}{\mu_+}}$ and $(q_+,r_+)\in\Gamma$. Then, $1+\frac\alpha N=\frac2{\mu_+}+\frac{2p}{r_+}$ and using H\"older and Hardy-Littlewood-Sobolev inequalities, one gets for $\theta_+:=q'_+-2$,
\begin{eqnarray*}
(C)
&\lesssim&\|(I_\alpha*|\cdot|^b|u|^p)|x|^b|u|^{p-1}\|_{L^{q_+'}_T(L^{r'_+}(|x|>1))}\\\\
&\lesssim&\||x|^b\|_{L^{\mu_+}(|x|>1)}^2\|\|u\|_{r_+}^{2p-1}\|_{L^{q'_+}(0,T)}\\
&\lesssim&\|u\|_{L^\infty_T(L^{r_+})}^{2(p-1)-\theta_+}\|u\|_{L^{q_+}_T(L^{r_+})}^{1+\theta_+},
\end{eqnarray*}
Since $p>p_*$, it follows that $(2p-1)q'_+>q_+$ and 
$$0<\theta_+<2(p-1).$$
Moreover, $p_*<p<p^*$ implies that $2<r_+<\frac{2N}{N-2}.$ The estimation of $(D)$ follows arguing as in $(B)$. This finishes the proof.
%\end{proof}
%%%%%%%%%%%%%%%%%%%%%%%%%%%%%%%%%%%%%%%%%%%%%%%%%%%%%%%%%%%%%%%%%%%%%%%%%%%%%%%%%%%%%%%%%%%%%%%%%%%%%%%%%%%%%%%%%%%%%%%%%%%%

%\end{thebibliography}

\end{document}